\documentclass{amsart}

\usepackage{geometry}
\usepackage{fancyhdr}
\usepackage{setspace}
\usepackage[all]{xy}
\usepackage{amssymb}
\usepackage{mathtools}
\usepackage[alphabetic,nobysame]{amsrefs}
\usepackage{amsfonts}
\usepackage{amsmath}
\usepackage{amsthm}
\usepackage{a4wide}
\usepackage{textcomp}
\usepackage{epsfig}
\usepackage{tikz-cd}
\usepackage{bm}
\usepackage[colorinlistoftodos]{todonotes}
\usepackage{aliascnt}
\usepackage[colorlinks, linkcolor=black,citecolor=black, urlcolor=black]{hyperref} 

\allowdisplaybreaks
\DeclareMathOperator{\gr}{gr}
\DeclareMathOperator{\pr}{pr}

\DeclareMathOperator{\im}{im}

\DeclareMathOperator{\ind}{ind}

\DeclareMathOperator{\id}{id}
\DeclareMathOperator{\can}{can}

\DeclareMathOperator{\sym}{sym}
\DeclareMathOperator{\fl}{flat}

\newcommand{\R}{\mathbb{R}}

\newcommand{\mfKill}[1]{}

\newtheorem{thm}{Theorem}[section]

\newaliascnt{prop}{thm}
\newtheorem{prop}[prop]{Proposition}
\aliascntresetthe{prop}

\newaliascnt{defi}{thm}

\aliascntresetthe{defi}

\newaliascnt{ex}{thm}

\aliascntresetthe{ex}

\newaliascnt{rmk}{thm}
\newtheorem{rmk}[rmk]{Remark}
\aliascntresetthe{rmk}

\newaliascnt{lemma}{thm}
\newtheorem{lemma}[lemma]{Lemma}
\aliascntresetthe{lemma}

\newaliascnt{cor}{thm}

\aliascntresetthe{cor}

\begin{document}

\title[On flat translated chains of contactomorphisms]{On flat translated chains\\
of contactomorphisms of $\mathbb{R}^{2n+1}$ and $\mathbb{R}^{2n} \times S^1$}

\author{Marco Mazzucchelli}
\address{Marco Mazzucchelli\newline\indent Sorbonne Université, Université Paris Cité, CNRS, IMJ-PRG, F-75005 Paris, France}
\email{marco.mazzucchelli@imj-prg.fr}

\author{Sheila Sandon}
\address{Sheila Sandon\newline\indent Universit\'e de Strasbourg, CNRS, IRMA UMR 7501, F-67000 Strasbourg, France}
\email{sandon@math.unistra.fr}

\begin{abstract}
\noindent We introduce the notions
of flat translated chains of contactomorphisms
and periodic flat translated chains of finite sequences of contactomorphisms,
and extend to these notions the theorem of Viterbo (1992)
on the multiplicity of periodic points
of compactly supported Hamiltonian diffeomorphisms
of $\mathbb{R}^{2n}$.
More precisely,
we show that every non-trivial compactly supported contactomorphism
of either $\mathbb{R}^{2n+1}$ or $\mathbb{R}^{2n} \times S^1$
that is contact isotopic to the identity
has infinitely many geometrically distinct flat translated chains
in the interior of the support
with respect to the standard contact form,
that the growth-rate of such translated chains
is at least linear
if the contactomorphism is non-negative,
as well as similar statements
for periodic flat translated chains
of finite sequences of contactomorphisms.
\end{abstract}

\maketitle


\section{Introduction}

Consider the Euclidean space $\mathbb{R}^{2n}$,
with coordinates $(x_1, y_1, \cdots, x_n, y_n)$,
endowed with its standard symplectic form
$\omega_0 = \sum_{j=1}^n dx_j \wedge dy_j$.
Using generating functions,
in \cite{Viterbo} Viterbo defined spectral selectors $c_-$ and $c_+$
for compactly supported Hamiltonian diffeomorphisms
of $(\mathbb{R}^{2n}, \omega_0)$
and proved, among several remarkable applications,
the following dynamical result.

\begin{thm}[Periodic points \cite{Viterbo}]\label{theorem: Viterbo}
Every non-trivial compactly supported Hamiltonian diffeomorphism
of $(\mathbb{R}^{2n}, \omega_0)$
has infinitely many geometrically distinct periodic points
in the interior of the support.
Moreover,
if such a diffeomorphism $\varphi$ is non-negative
(i.e.\ is the time-$1$ map of the flow
of a compactly supported non-negative Hamiltonian function)
then the number $N (\varphi, m)$
of periodic point of $\varphi$
in the interior of the support
of period smaller then or equal to $m$
grows at least linearly in $m$,
i.e.
\[
\limsup_{m \to \infty} \, \frac{N (\varphi, m)}{m} > 0 \,.
\]
\end{thm}

Consider now the odd dimensional Euclidean space $\mathbb{R}^{2n+1}$,
with coordinates $(x_1, y_1, \cdots, x_n,y_n, \theta)$,
endowed with the contact structure $\xi_0$
given by the kernel of the standard contact form
$\alpha_0 = d\theta + \sum_{j=1}^n \frac{x_jdy_j - y_jdx_j}{2}$,
and its quotient
$\mathbb{R}^{2n} \times S^1$
(where $S^1 = \mathbb{R}/\mathbb{Z}$),
endowed with the induced contact form and contact structure,
still denoted by $\alpha_0$ and $\xi_0$.
Extending the constructions of Viterbo to the contact case,
Bhupal \cite{Bhupal} and the second author \cites{San10, San11}
defined spectral selectors $c_-$ and $c_+$
for compactly supported contactomorphisms
of $(\mathbb{R}^{2n+1}, \xi_0)$
and $(\mathbb{R}^{2n} \times S^1, \xi_0)$
contact isotopic to the identity,
and used them to obtain generalizations
of some of the results of \cite{Viterbo}:
the definition of a bi-invariant partial order
on the group of compactly supported contactomorphisms
of $(\mathbb{R}^{2n+1}, \xi_0)$ and $(\mathbb{R}^{2n} \times S^1, \xi_0)$
contact isotopic to the identity,
of an integer-valued bi-invariant metric on this group
in the case of $(\mathbb{R}^{2n} \times S^1, \xi_0)$,
and of an integer-valued capacity for domains
of $(\mathbb{R}^{2n} \times S^1, \xi_0)$,
allowing to obtain an alternative proof
of the contact non-squeezing theorem for integers
of Eliashberg, Kim and Polterovich \cite{EKP}.

In 2010 Viktor Ginzburg asked to the second author
whether the spectral selectors $c_-$ and $c_+$ for contactomorphisms
could also be used to obtain a contact analogue
of Theorem \ref{theorem: Viterbo}.
In order to answer this question,
the second author introduced in \cite{San12}
the notions of translated points
and iterated translated points
of contactomorphisms,
and proved that
every compactly supported contactomorphism
of $(\mathbb{R}^{2n+1}, \xi_0)$ or $(\mathbb{R}^{2n} \times S^1, \xi_0)$
contact isotopic to the identity
that is positive
(i.e.\ is the time-$1$ map of the flow
of a Hamiltonian function with non-empty support
that is positive in the interior of the support)
has infinitely many geometrically distinct
iterated translated points with respect to $\alpha_0$
in the interior of the support.

Recall that a translated point
of a contactomorphism $\phi$
of a contact manifold $(M, \xi)$
with respect to a contact form $\alpha$ for $\xi$
is a point $p$ of $M$
such that $p$ and $\phi(p)$ are in the same orbit
of the Reeb flow of $\alpha$
and $g(p) = 0$,
where $g$ is the conformal factor of $\phi$
with respect to $\alpha$,
i.e.\ the function satisfying
$\phi^{\ast}\alpha = e^g \alpha$.
Recall also that in \cite{San12}
a point $p$ of $\mathbb{R}^{2n+1}$ or $\mathbb{R}^{2n} \times S^1$
is called an iterated translated point
of a contactomorphism $\phi$
if it is a translated point of some iteration $\phi^m$.
Translated points of $\mathcal{C}^0$-small contactomorphisms
of compact contact manifolds
always exist,
and satisfy lower bounds analogous to those
predicted by the Arnold conjecture
for fixed points of Hamiltonian diffeomorphisms
\cite{San11b}.
Existence and multiplicity results
beyond the $\mathcal{C}^0$-small case
have also been proved on several classes of contact manifolds
\cites{San12, San11b,
AM, AFM, Shelukhin, MN, MU, GKPS, Tervil, Allais22a, Allais22b, AArl},
but compact contact manifolds
admitting contactomorphisms
contact isotopic to the identity
without translated points also exist
\cites{Cant, HS}.
The known results on existence of translated points
seem to suggest a link of this problem,
still to be understood,
with the orderability question for contact manifolds
introduced by Eliashberg and Polterovich \cite{EP},
and confirm the interest of translated points
for exploring global rigidity phenomena
in contact topology.
On the other hand,
iterated translated points have not been studied further,
and with hindsight do not seem to be a very relevant notion.
An important difference with respect to periodic
(or iterated\footnote{Fixed points of iterations
of a Hamiltonian diffeomorphism
are alternatively called periodic or iterated points.}) points
of Hamiltonian diffeomorphisms
is that iterated translated points are not periodic,
in the sense that if $p$ is a translated point
of a contactomorphism $\phi$
then $p$ is in general not also a translated point
of the iterations $\phi^{m}$.
Partly because of this,
although the spectral selectors $c_{\pm}$ satisfy in the contact case
properties similar to those that hold in the symplectic case,
the result in \cite{San12}
is weaker than the one of \cite{Viterbo}.
On the other hand,
as we now discuss,
a periodicity property 
(Lemma \ref{lemma: action relation for not geometrically distinct chains} below)
similar to that of periodic points
of Hamiltonian diffeomorphisms
holds for (flat) translated chains of contactomorphisms,
making it possible to obtain for these objects
a result that is totally analogous
to the one of \cite{Viterbo}.

The notion of \emph{translated chains}
of contactomorphisms has been introduced in \cite{FSZ},
in order to prove with generating functions
the general contact non-squeezing theorem
in $(\mathbb{R}^{2n} \times S^1, \xi_0)$
of Chiu \cite{Chiu} and Fraser \cite{Fraser}.
We give in the present work
a more general definition of translated chains,
and call the translated chains of \cite{FSZ}
\emph{symmetric} translated chains.
Let $\phi$ be a contactomorphism
of a contact manifold $\big(M, \xi = \ker (\alpha)\big)$,
and denote by $g$ its conformal factor
with respect to $\alpha$.
Denote by $\{\varphi_t^{\alpha}\}$ the Reeb flow of $\alpha$,
and let $k$ be a positive integer.
A \emph{translated} $k$\emph{-chain} of $\phi$
with respect to $\alpha$
is a $k$-tuple $\bm{p} = (p_0, \cdots, p_{k-1})$
of points such that
\[
g(p_0) + \dots + g (p_{k-1}) = 0
\]
and
\[
p_{j+1} = \varphi_{-t_j}^{\alpha_0} \circ \phi \, (p_j)
\quad \text{ for } j \in \mathbb{Z}_k \,,
\]
for a $k$-tuple $\bm{t} = (t_0, \cdots, t_{k-1})$ of real numbers.
More generally,
a \emph{translated chain}
with respect to $\alpha$
of a sequence $\bm{\phi} = (\phi_0, \cdots, \phi_{k-1})$
of contactomorphisms
is a $k$-tuple $\bm{p} = (p_0, \cdots, p_{k-1})$
of points such that
\[
g_0 (p_0) + \dots + g_{k-1} (p_{k-1}) = 0 \,,
\]
where $g_j$ denotes the conformal factor of $\phi_j$,
and
\[
p_{j+1} = \varphi_{-t_j}^{\alpha_0} \circ \phi_j \, (p_j)
\quad \text{ for } j \in \mathbb{Z}_k \,,
\]
for a $k$-tuple $\bm{t} = (t_0, \cdots, t_{k-1})$ of real numbers.
In particular,
a translated $k$-chain of a contactomorphism $\phi$
is a translated chain of the sequence
\[
(\phi)^k := (\underbrace{\phi, \cdots, \phi}_{k}) \,.
\]
If $t_0 = \cdots = t_{k-1}$
then we say that $\bm{p}$
is a \textit{symmetric translated chain}.

For $k = 1$ the notion of translated chain
reduces to the one of translated point.
For $k \geq 2$,
the existence problem of translated chains
if often underdetermined.
Indeed,
let $\bm{\phi}$ be a sequence of contactomorphisms,
and consider the map 
\[
\iota: M \times \R^k \to M \times M \times \R \;,\quad
\iota \, (p_0, \bm{t})
= \big(p_0 \,,\, p_k \,,\, g_0 (p_0) + \cdots + g_{k-1} (p_{k-1})\big) \,,
\] 
where $p_{j+1} = \varphi_{-t_j}^{\alpha_0} \circ \phi_j \, (p_j)$
for $j = 0, \cdots, k-1$.
Then $(p_0, \cdots, p_{k-1})$ is a translated chain of $\bm{\phi}$ 
if and only if $\iota (p_0, \bm{t})$
belongs to the intersection of the image of $\iota$
with the submanifold $\Delta = \{\, (p,p,0) \;\lvert\; p\in M \,\}$.
If $k \geq 2$ and $\iota$ is transverse to $\Delta$,
the above intersection (if not empty)
has positive dimension.
On the other hand,
symmetric translated chains
are generically isolated,
since they correspond to the intersections with $\Delta$
of the map
\[
\iota_{\sym}: M \times \R \to M \times M \times \R \;,\quad
\iota_{\sym} \, (p_0, t)
= \big(p_0 \,,\, p_k \,,\, g_0 (p_0) + \cdots + g_{k-1} (p_{k-1})\big) \,,
\] 
where $p_{j+1} = \varphi_{-t}^{\alpha_0} \circ \phi_j \, (p_j)$
for $j = 0, \cdots, k-1$.

We now introduce another class of translated chains
that are generically isolated.
We say that a $k$-tuple
$\bm{p} = (p_0, \cdots, p_{k-1})$
of points of $M$
is a \emph{flat translated chain}
of a sequence $\bm{\phi} = (\phi_0, \cdots, \phi_{k-1})$
of contactomorphisms of $\big(M, \xi = \ker (\alpha) \big)$
with respect to $\alpha$
if it is a translated chain of $\bm{\phi}$
with respect to $\alpha$
such that $g_j (p_j) = 0$ for all $j$,
where $g_j$ is the conformal factor of $\phi_j$.
Flat translated chains are generically isolated,
since they correspond to the intersections with
$\Delta' = \{\, (p, p, 0, \cdots, 0) \;\lvert\; p\in M \,\}$
of the image of the map
\[
\iota_{\fl}: M \times \R^k \to M \times M \times \R^k \;,\quad
\iota_{\fl} \, (p_0, \bm{t})
= \big(p_0 \,,\, p_k \,,\, g_0 (p_0) \,,\, \cdots \,,\, g_{k-1} (p_{k-1})\big) \,,
\]
where $p_{j+1} = \varphi_{-t_j}^{\alpha_0} \circ \phi_j \, (p_j)$
for $j = 0, \cdots, k-1$.

\begin{rmk}\label{remark: flat translated chains}
The motivation for the terminology
of flat translated chains
comes from the fact that,
seen in the symplectization,
all the nodes of a flat translated chain
are at the same level.
More precisely,
recall that the lift of a contactomorphism $\phi$
of a contact manifold $\big(M, \xi = \ker (\alpha)\big)$
to the symplectization
$\big(\mathbb{R} \times M \,,\, d (e^r \alpha)\big)$
is the symplectomorphism $\widetilde{\phi}$
defined by $\widetilde{\phi} \, (r, p) = \big( r + g(p) \,,\, \phi(p) \big)$,
where $g$ is the conformal factor of $\phi$
with respect to $\alpha$.
We say that a $k$-tuple
$\big( (r_0, p_0), \cdots, (r_{k-1}, p_{k-1}) \big)$
of points of $\mathbb{R} \times M$
is a \emph{leafwise chain} of a sequence
$\bm{\varphi} = (\varphi_0, \cdots, \varphi_{k-1})$
of Hamiltonian diffeomorphisms
if $(r_{j+1}, p_{j+1}) =
\widetilde{\varphi_{-t_j}^{\alpha}} \circ \varphi \, (p_j)$
for $j \in \mathbb{Z}_k$.
If the sequence $\bm{\varphi}$ is of the form
$(\widetilde{\phi_0}, \cdots, \widetilde{\phi_{k-1}})$
for a sequence $\bm{\phi} = (\phi_0, \cdots, \phi_{k-1})$
of contactomorphisms,
then the projection $(p_0, \cdots, p_{k-1})$
of a leafwise chain
$\big( (r_0, p_0), \cdots, (r_{k-1}, p_{k-1}) \big)$
of $\bm{\varphi}$
is a translated chain of $\bm{\phi}$
with respect to $\alpha$,
which is flat if and only if $r_0 = \cdots = r_{k-1}$.
\end{rmk}

For each positive integer $m$,
we define the $m$-th iteratation of a sequence
$\bm{\phi} = (\phi_0, \cdots, \phi_{k-1})$
of contactomorphisms of $\big(M, \xi = \ker(\alpha)\big)$
to be the sequence
\[
\bm{\phi}^m =
(\underbrace{\bm{\phi}, \cdots, \bm{\phi}}_{m}) \,.
\]
We say that $\bm{p}$ is a
\emph{periodic (flat) translated chain}
of $\bm{\phi}$ with respect to $\alpha$
if it is a (flat) translated chain of $\bm{\phi}^m$
with respect to $\alpha$
for some positive integer $m$,
which we call the \emph{period}
of the periodic (flat) translated chain $\bm{p}$.
In particular,
a periodic (flat) translated chain of period $m$
of a contactomorphism $\phi$
is just a (flat) translated $m$-chain of $\phi$.

We consider the equivalence relation
on the set of periodic flat translated chains
of a sequence $\bm{\phi}$
that is generated by the following two operations:
\begin{itemize}
\item \emph{Cyclic permutation}:
$(\bm{p}_0, \cdots, \bm{p}_{m-1}) \mapsto
(\bm{p}_1, \cdots, \bm{p}_{m-1}, \bm{p}_0)$
for a periodic flat translated chain
$(\bm{p}_0, \cdots, \bm{p}_{m-1})$
of period $m$.
\item \emph{Iteration}:
$\bm{p} \mapsto \bm{p}^l = (\underbrace{\bm{p}, \cdots, \bm{p}}_{l})$.
\end{itemize}
We say that two periodic (flat) translated chains of $\bm{\phi}$
are geometrically distinct
if they are not in the same class for this equivalence relation.

Before stating our main result,
we still need to recall or introduce a few more notions.
Recall first that there is a partial order
on the group of compactly supported contactomorphisms
of either $(\mathbb{R}^{2n+1}, \xi_0)$ \cite{Bhupal}
or $(\mathbb{R}^{2n} \times S^1, \xi_0)$
\cite{San11}
contact isotopic to the identity,
defined by posing $\phi \leq \psi$
if $\psi \circ \phi^{-1}$ is the time-$1$ map
of the flow of a non-negative Hamiltonian function.
We define a partial order on the space of sequences
of compactly supported contactomorphisms
of either $(\mathbb{R}^{2n+1}, \xi_0)$
or $(\mathbb{R}^{2n} \times S^1, \xi_0)$
contact isotopic to the identity
by posing $\bm{\phi} \leq \bm{\psi}$,
for $\bm{\phi} = (\phi_0, \cdots, \phi_{k-1})$
and $\bm{\psi} = (\psi_0, \cdots, \psi_{k-1})$,
if $\phi_j \leq \psi_j$ for all $j$.
We say that $\bm{\phi}$ is non-negative
if $\bm{\id} \leq \bm{\phi}$,
where $\bm{\id} = (\id, \cdots, \id)$.
We say that a sequence $\bm{\phi} = (\phi_0, \cdots, \phi_{k-1})$
is non-trivial
if at least one of the $\phi_j$'s
is not the identity.
Finally,
we say that a periodic flat translated chain
$\bm{p} = (p_0, \cdots, p_{mk-1})$ of a non-trivial sequence
$\bm{\phi} = (\phi_0, \cdots,\phi_{k-1})$
is contained in the interior of the support of $\bm{\phi}$
if $p_j$ is contained in the interior of the support
of  $\phi_{j\,\mathrm{mod}\,k}$
for at least one value of $j \in \mathbb{Z}_{mk}$,
and we denote by $N(\bm{\phi}, m)$
the number of geometrically distinct
periodic flat translated chains of $\bm{\phi}$
contained in the interior of the support
and having period smaller than or equal to $m$.
Our main result is then the following analogue
of \autoref{theorem: Viterbo}.

\begin{thm}[Periodic flat translated chains]\label{theorem: main}
Every non-trivial sequence $\bm{\phi}$
of compactly supported contactomorphisms
of either $(\mathbb{R}^{2n+1}, \xi_0)$ or $(\mathbb{R}^{2n} \times S^1, \xi_0)$
contact isotopic to the identity
has infinitely many geometrically distinct
periodic flat translated chains
with respect to $\alpha_0$
in the interior of the support.
Moreover,
if $\bm{\phi}$ is non-negative
then the growth of such periodic flat translated chains
is at least linear,
i.e.
\[
\limsup_{m \to \infty} \, \frac{N (\bm{\phi}, m)}{m} > 0 \,.
\]
\end{thm}

The proof of Theorem \ref{theorem: main}
is parallel to the one of Theorem \ref{theorem: Viterbo}
given in \cite{Viterbo}.
One of the key ingredients for the latter
is the following periodicity property
of periodic points of Hamiltonian diffeomorphisms
and their actions.
Recall that the action
of a $m$-periodic point $p$
of a compactly supported Hamiltonian diffeomorphism
$\varphi$ of $(\mathbb{R}^{2n}, \omega_0)$
is the symplectic action $\mathcal{A}_{\varphi^m}(p)$
of $p$ as a fixed point of $\varphi^m$.
If $p$ is a $m$-periodic point of $\varphi$
then, for every positive integer $l$,
$p$ is also a $lm$-periodic point of $\varphi$,
and we have
$\mathcal{A}_{\varphi^{lm}} (p) = l \, \mathcal{A}_{\varphi^m}(p)$.
This implies that if $p_1$ and $p_2$
are a $m_1$-periodic and a $m_2$-periodic point respectively
of $\varphi$ with
\[
\frac{1}{m_1} \; \mathcal{A}_{\varphi^{m_1}} (p_1)
\neq \frac{1}{m_2} \; \mathcal{A}_{\varphi^{m_2}} (p_2)    
\]
then $p_1$ and $p_2$ are geometrically distinct.
The analogue of this property
for periodic flat translated chains is as follows.
The action of a translated chain $\bm{p} = (p_0, \cdots, p_{k-1})$
of a sequence $\bm{\phi} = (\phi_0, \cdots, \phi_{k-1})$
of contactomorphisms of $(\mathbb{R}^{2n+1}, \xi_0)$
with respect to $\alpha_0$
is defined by
$$
\mathcal{A}_{\bm{\phi}} (\bm{p}) = \sum_{j=0}^{k-1} t_j \,,
$$
where $\bm{t} = (t_0, \cdots, t_{k-1})$
is the unique $k$-tuple 
of real numbers
(called the \emph{time-shifts} of $\bm{p}$)
such that
$p_{j+1} = \varphi_{-t_j}^{\alpha_0} \circ \phi \, (p_j)$.
Consider now a sequence $\bm{\phi} = (\phi_0, \cdots, \phi_{k-1})$
of compactly supported contactomorphisms
of $(\mathbb{R}^{2n} \times S^1, \xi_0)$
contact isotopic to the identity,
and denote by
$\bm{\phi}_{\mathbb{R}}
= \big( (\phi_0)_{\mathbb{R}}, \cdots, (\phi_{k-1})_{\mathbb{R}} \big)$
its lift to $(\mathbb{R}^{2n+1}, \xi_0)$,
i.e.\ each $(\phi_j)_{\mathbb{R}}$ is the contactomorphism
of $(\mathbb{R}^{2n+1}, \xi_0)$
that projects to $\phi_j$ and is the identity
over the complement of the support of $\phi_j$.
We define the action of a translated chain $\bm{p}$
of $\bm{\phi}$ by
$$
\mathcal{A}_{\bm{\phi}} (\bm{p})
= \mathcal{A}_{\bm{\phi}_{\mathbb{R}}} (\bm{p}_{\mathbb{R}})
$$
for any $k$-tuple $\bm{p}_{\mathbb{R}}$
of points of $\mathbb{R}^{2n+1}$
projecting to $\bm{p}$.
The following lemma,
which follows directly from the definitions,
will play an important role
in the proof of Theorem \ref{theorem: main}.

\begin{lemma}[Periodicity]
\label{lemma: action relation for not geometrically distinct chains}
Let $\bm{\phi}$ be a sequence
of compactly supported contactomorphisms
of either $(\mathbb{R}^{2n+1}, \xi_0)$ or $(\mathbb{R}^{2n} \times S^1, \xi_0)$
contact isotopic to the identity,
and $\bm{p}_1$ and $\bm{p}_2$
periodic translated chains of $\bm{\phi}$
of periods $m_1$ and $m_2$ respectively.
Assume that
\[
\frac{1}{m_1} \; \mathcal{A}_{\bm{\phi}^{m_1}} (\bm{p}_1)
\neq \frac{1}{m_2} \; \mathcal{A}_{\bm{\phi}^{m_2}} (\bm{p}_2) \,.
\]
Then $\bm{p}_1$ and $\bm{p}_2$
are geometrically distinct.
\end{lemma}

In order to prove Theorem \ref{theorem: main}
in the case of $(\mathbb{R}^{2n} \times S^1, \xi_0)$,
we define spectral selectors $c_-$ and $c_+$
for sequences of compactly supported contactomorphisms
contact isotopic to the identity,
detecting flat translated chains.
To do so,
given a sequence $\bm{\phi} = (\phi_0, \cdots, \phi_{k-1})$
with $k$ odd
and special generating functions quadratic at infinity
$F_j: \mathbb{R}^{2n+1} \times \mathbb{R}^{N_j} \rightarrow \mathbb{R}$
for the $\phi_j$'s,
we define a function
$F_0 \,\sharp\, \cdots \,\sharp\, F_{k-1}:
\mathbb{R}^{(2n+1)k} \times \mathbb{R}^{N_0 + \cdots + N_{k-1}}
\rightarrow \mathbb{R}$
whose critical points are in bijection
with the flat translated chains of $\bm{\phi}_{\mathbb{R}}$,
with critical values given by the actions 
(Proposition \ref{proposition: composition formula contact flat}).
We compactify $F_0 \,\sharp\, \cdots \,\sharp\, F_{k-1}$
to a function
$$
\overline{F_0 \,\sharp\, \cdots \,\sharp\, F_{k-1}}:
S^{2n} \times \mathbb{T}^k \times \mathbb{R}^{2n(k-1)}
\times \mathbb{R}^{N_0 + \cdots + N_{k-1}} \rightarrow \mathbb{R}
$$
that is quadratic at infinity over $S^{2n} \times \mathbb{T}^k$
(Proposition \ref{proposition: quadratic at infinity flat chains}).
The set of critical values
of $\overline{F_0 \,\sharp\, \cdots \,\sharp\, F_{k-1}}$
is then equal to the (flat) action spectrum of $\bm{\phi}$,
i.e.\ the set $\mathcal{A}_{\bm{\phi}}$
of actions $\mathcal{A}_{\bm{\phi}} (\bm{p})$
of all the flat translated chains $\bm{p}$ of $\bm{\phi}$.
We define $c_- (\bm{\phi})$ and $c_+ (\bm{\phi})$
to be the spectral selectors
of $\overline{F_0 \,\sharp\, \cdots \,\sharp\, F_{k-1}}$
with respect to the unit class
and the dual of the fundamental class
of $S^{2n} \times \mathbb{T}^k$
respectively.
For a sequence $\bm{\phi} = (\phi_0, \cdots, \phi_{k-1})$
with $k$ even
we then define $c_{\pm} (\bm{\phi}) = c_{\pm} (\bm{\phi}, \id)$.

We say that a sequence
$\bm{\phi} = (\phi_0, \cdots, \phi_{k-1})$
is \emph{dynamically trivial}
if each $\phi_j$ is a strict contactomorphism
(i.e.\ it preserves the contact form $\alpha_0$)
and $\phi_{k-1} \circ \cdots \circ \phi_0 = \id$.
In Section \ref{section: spectral selectors}
we prove the following properties
of the spectral selectors $c_{\pm}$.

\begin{prop}\label{proposition: properties spectral selectors intro}
The spectral selectors $c_{\pm}$ for sequences
of compactly supported contactomorphisms
of $(\mathbb{R}^{2n} \times S^1, \xi_0)$
contact isotopic to the identity
satisfy the following properties:
\renewcommand{\theenumi}{\roman{enumi}}
\begin{enumerate}
\item \label{properties spectral sel intro: spectrality} \emph{Spectrality:}
$c_{\pm} (\bm{\phi}) \in \mathcal{A}_{\bm{\phi}}$,
in particular $c_{\pm} (\bm{\id}) = 0$.
\item \label{properties spectral sel intro: sign} \emph{Sign:}
$c_- (\bm{\phi}) \leq 0$ and $c_+ (\bm{\phi}) \geq 0$.
\item \label{properties spectral sel intro: non-degeneracy} \emph{Non-degeneracy:}
$c_- (\bm{\phi}) = c_+ (\bm{\phi}) = 0$
if and only if $\bm{\phi}$ is dynamically trivial.
\item \label{properties spectral sel intro: monotonicity} \emph{Monotonicity:}
if $\bm{\phi} \leq \bm{\psi}$
then $c_{\pm}(\bm{\phi}) \leq c_{\pm} (\bm{\psi})$.
\item \label{properties spectral sel intro: decreasing sequence}
\emph{Stabilization:}
$c_{\pm} (\bm{\phi}, \id) = c_{\pm} (\bm{\phi})$.
\item \label{properties spectral sel intro: relation symplectic}
\emph{Relation with the symplectic spectral selectors:}
For every sequence $\bm{\varphi} = (\varphi_0, \cdots, \varphi_{k-1})$
of compactly supported Hamiltonian diffeomorphisms
of $(\mathbb{R}^{2n}, \omega_0)$,
we have
$$
c_{\pm} (\widetilde{\bm{\varphi}})
= c_{\pm} (\varphi_{k-1} \circ \cdots \circ \varphi_0) \,,
$$
where $\widetilde{\bm{\varphi}}$ is the sequence
formed by the lifts of the $\varphi_j$'s
to $(\mathbb{R}^{2n} \times S^1, \xi_0)$.
\end{enumerate}
\end{prop}

Using (\ref{properties spectral sel intro: monotonicity})
and (\ref{properties spectral sel intro: relation symplectic})
and the analogous statement for spectral selectors
of Hamiltonian diffeomorphisms
(Lemma \ref{lemma: bound symplectic}),
in Section \ref{section: spectral selectors}
we deduce the following result.

\begin{lemma}[Upper bound]\label{lemma: bound contact}
For any sequence $\bm{\phi}$
of compactly supported contactomorphisms of $(\mathbb{R}^{2n} \times S^1, \xi_0)$
contact isotopic to the identity,
the sets $\{\, c_+ (\bm{\phi}^m) \,,\,
m \in \mathbb{Z}_{> 0} \,\}$
and $\{\, \lvert \, c_- (\bm{\phi}^m) \,\lvert \,,\,
m \in \mathbb{Z}_{> 0} \,\}$
are bounded from above.
\end{lemma}

The main lines of the proof of Theorem \ref{theorem: main}
in the case of $(\mathbb{R}^{2n} \times S^1, \xi_0)$
are then as follows.
If a sequence $\bm{\phi} = (\phi_0, \cdots, \phi_{k-1})$ is dynamically trivial
then it has infinitely many geometrically distinct flat translated chains
in the interior of the support,
indeed for any point $(z_0, \theta_0)$ of $\mathbb{R}^{2n} \times S^1$
the $k$-tuple
\[
\big( (z_0, \theta_0) \,,\, \phi_0 \, (z_0, \theta_0) \,,\, \cdots \,,\,
\phi_{k-2} \circ \cdots \circ \phi_0 \, (z_0, \theta_0) \big)
\]
is a flat translated chain of $\bm{\phi}$.
Since the iterations $\bm{\phi}^m$
of a dynamically trivial sequence $\bm{\phi}$
are also dynamically trivial,
the conclusions of Theorem \ref{theorem: main}
are trivially satisfied in this case.
Assume now that $\bm{\phi}$ is not dynamically trivial,
and denote by $M_{\bm{\phi}}$ the set of positive integers $m$
such that $\bm{\phi}^m$ is not dynamically trivial.
As we will see,
$M_{\bm{\phi}}$ is an infinite set.
Suppose also that $\bm{\phi}$ is non-negative.
Then (\ref{properties spectral sel intro: spectrality}),
(\ref{properties spectral sel intro: sign}),
(\ref{properties spectral sel intro: non-degeneracy})
and (\ref{properties spectral sel intro: monotonicity})
imply that $c_+ (\bm{\phi}^m) > 0$
for all $m \in M_{\bm{\phi}}$.
By (\ref{properties spectral sel intro: spectrality}),
for any $m \in M_{\bm{\phi}}$ there is a flat translated chain
$\bm{p}_m$ of $\bm{\phi}^m$ with
$\mathcal{A}_{\bm{\phi}^m} (\bm{p}_m) = c_+ (\bm{\phi}^m)$.
As in \cite{Viterbo},
Lemmas \ref{lemma: action relation for not geometrically distinct chains}
and \ref{lemma: bound contact}
then formally imply that infinitely many
of the translated chains $\bm{p}_m$
are geometrically distinct.
Moreover,
we will see that
property (\ref{properties spectral sel intro: decreasing sequence})
allows to obtain the result on linear growth.
For a general (i.e.\ not necessarily non-negative)
sequence $\bm{\phi}$,
the existence of infinitely many geometrically distinct
periodic flat translated chains
is obtained by an argument similar to the one in the non-negative case,
using that there exists a sequence $l(j) \to \infty$ in $M_{\bm{\phi}}$
such that either $c_+ (\bm{\phi}^{l(j)}) > 0$ for all $j$
or $c_- (\bm{\phi}^{l(j)}) < 0$ for all $j$
(which follows from the property
(\ref{properties spectral sel intro: non-degeneracy})).
We refer to Section \ref{section: proof}
for the complete arguments,
as well as for the observation that
Theorem \ref{theorem: main}
in the case of $(\mathbb{R}^{2n+1}, \xi_0)$
can easily be deduced
from the case of $(\mathbb{R}^{2n} \times S^1, \xi_0)$.

\begin{rmk}\label{remark: symmetric translated chains}
In \cite{FSZ},
the authors defined functions
detecting \emph{symmetric} translated chains
of compactly supported contactomorphisms
of $(\mathbb{R}^{2n} \times S^1, \xi_0)$
contact isotopic to the identity,
and used them to construct
$\mathbb{Z}_k$-equivariant homology groups
for contact isotopies and domains
of $(\mathbb{R}^{2n} \times S^1, \xi_0)$.
Using similar functions
for sequences of contactomorphisms,
we expect that it should be possible to define
spectral selectors detecting symmetric translated chains,
and use them to show that
every non-trivial sequence of compactly supported contactomorphisms
of $(\mathbb{R}^{2n+1}, \xi_0)$ or $(\mathbb{R}^{2n} \times S^1, \xi_0)$
contact isotopic to the identity
has infinitely many geometrically distinct
periodic symmetric translated chains
with respect to $\alpha_0$
in the interior of the support.
However,
in order to obtain a result on linear growth
analogous to the one in Theorem \ref{theorem: main}
we would need an analogue
of Proposition \ref{proposition: properties spectral selectors intro}
(\ref{properties spectral sel intro: decreasing sequence}),
but it is not clear to us whether this property
can be expected to hold
for spectral selectors detecting symmetric translated chains.
Indeed,
as we will see,
the proof of Proposition \ref{proposition: properties spectral selectors intro}
(\ref{properties spectral sel intro: decreasing sequence})
relies on the fact that,
for any sequence $\bm{\phi}$,
if $(\bm{p}, p)$ is a (flat) translated chain of $(\bm{\phi}, \id)$
then $\bm{p}$ is a (flat) translated chain of $\bm{\phi}$.
However,
if $(\bm{p}, p)$ is a symmetric translated chain
of $(\bm{\phi}, \id)$
then the translated chain $\bm{p}$ of $\bm{\phi}$
is in general not symmetric.
\end{rmk}

Periodic points of Hamiltonian diffeomorphisms
are the object of important conjectures in symplectic topology,
as the Conley and Hofer--Zehnder conjectures.
The analogy of Theorem \ref{theorem: main}
on periodic flat translated chains
of sequences of contactomorphisms of $(\mathbb{R}^{2n+1}, \xi_0)$
or $(\mathbb{R}^{2n} \times S^1, \xi_0)$
with Theorem \ref{theorem: Viterbo}
on periodic points of Hamiltonian diffeomorphisms
of $(\mathbb{R}^{2n}, \omega_0)$
suggests that similar conjectures might also hold,
on certain contact manifolds,
also for periodic flat (or maybe symmetric) translated chains.
In particular,
the possibility of proving an analogue of the Hofer--Zehnder conjecture
for (periodic) translated chains
of (sequences of) contactomorphisms
of real projective spaces
(or more general lens spaces)
will be explored in a forthcoming work.

The article is organized as follows.
In Section \ref{section: gfqi}
we recall the properties of generating functions quadratic at infinity
that are needed in the rest of the article,
and describe a formula to obtain generating functions quadratic at infinity
for odd compositions of compactly supported Hamiltonian diffeomorphisms
of $(\mathbb{R}^{2n}, \omega_0)$.
In Section \ref{section: functions translated chains}
we describe an analogous formula
to obtain functions quadratic at infinity
detecting flat translated chains
of sequences of compactly supported contactomorphisms
of $(\mathbb{R}^{2n} \times S^1, \xi_0)$
contact isotopic to the identity.
In Section \ref{section: spectral selectors}
we use such functions
to define spectral selectors
detecting flat translated chains,
and prove Proposition \ref{proposition: properties spectral selectors intro}
and Lemma \ref{lemma: bound contact}.
In Section \ref{section: proof}
we use all these ingredients
to prove Theorem \ref{theorem: main}.

\subsection*{Acknowledgments}
The second author thanks Viktor Ginzburg
for his question of 16 years ago.
We both thank Viktor Ginzburg and Basak Gurel
for all the discussions we have had since then.


\section{Generating functions quadratic at infinity}\label{section: gfqi}

In this section we recall the notion of generating functions
for compactly supported
Hamiltonian diffeomorphisms of $(\mathbb{R}^{2n}, \omega_0)$
and compactly supported contactomorphisms
of $(\mathbb{R}^{2n} \times S^1, \xi_0)$
contact isotopic to the identity,
following \cites{Viterbo, Bhupal, San11}
and the references therein.
We also describe a formula
to obtain generating functions quadratic at infinity
for odd compositions of compactly supported
Hamiltonian diffeomorphisms of $(\mathbb{R}^{2n}, \omega_0)$,
following essentially \cite{Allais} and \cite{FSZ}.

A function $F: E \rightarrow \mathbb{R}$
defined on the total space of a trivial vector bundle
$E = B \times \mathbb{R}^N \rightarrow B$
is a \emph{generating function}
if the differential $dF: E \rightarrow T^{\ast}E$
is transverse to the subbundle $N_E^{\ast}$ of $T^{\ast}E$
formed by the covectors that vanish on vertical vectors.
In this case, the space
$$
\Sigma_F = dF^{-1} \big( N_E^{\ast} \cap \im (dF) \big)
$$
of fibre critical points of $F$
is a submanifold of $E$,
of dimension equal to the dimension of $B$.
We endow the cotangent bundle $T^{\ast}B$
with its canonical symplectic form
$\omega_{\can} = d \lambda_{\can}$.
The map
$$
i_F: \Sigma_F \rightarrow T^{\ast}B
$$
that associates to
$e = (q, \zeta) \in \Sigma_F \subset B \times \mathbb{R}^N$
the covector $i_F(e) \in T_q^{\ast}B$
defined by $i_F(e) (X) = dF (\widehat{X})$,
where $\widehat{X}$ is any vector in $T_eE$
projecting to $X$,
is then an exact Lagrangian immersion,
with $i_F^{\;\ast} \, \lambda_{\can} = d ( \left. F \right\lvert_{\Sigma_F} )$.
We endow the $1$-jet bundle $J^1B = T^{\ast}B \times \mathbb{R}$
with its canonical contact structure
$\xi_{\can} = \ker \, (dw - \lambda_{\can})$,
where $w$ denotes the coordinate in $\mathbb{R}$.
The lift of $i_F: \Sigma_F \rightarrow T^{\ast}B$
to $(J^1B, \xi_{\can})$
is the Legendrian immersion
\[
j_F: \Sigma_F \rightarrow J^1B \,,\;
e \mapsto \big( i_F(e), F(e) \big) \,.
\]
If $i_F: \Sigma_F \rightarrow T^{\ast}B$ is an embedding
we say that $F$ is a generating function
for the Lagrangian submanifold $\im (i_F)$
of $(T^{\ast}B, \omega_{\can})$.
The map $i_F$ then induces a bijection
between the critical points of $F$
and the intersections of $\im (i_F)$ with the zero section.
Similarly, if $j_F: \Sigma_F \rightarrow J^1B$ is an embedding
we say that $F$ is a generating function
for the Legendrian submanifold $\im (j_F)$
of $(J^1B, \xi_{\can})$.
The map $j_F$ then induces a bijection
between the critical points of $F$
and the Reeb chords between $\im (j_F)$ and the zero section.

Suppose now that $B$ is compact.
A generating function
$F: E = B \times \mathbb{R}^N \rightarrow \mathbb{R}$
is said to be \emph{quadratic at infinity}
if there exists a non-degenerate quadratic form $F_{\infty}$ on $E$,
i.e.\ a map $F_{\infty}: E \rightarrow \mathbb{R}$
whose restriction to every fibre
is a non-degenerate quadratic form,
such that
$\partial_v (F - F_{\infty}): E \rightarrow E^{\ast}$
is bounded,
where $\partial_v$ denotes the vertical derivative.
It has been proved by Sikorav \cite{Sikorav87}
that every Lagrangian submanifold of $(T^{\ast}B, \omega_{\can})$
Hamiltonian isotopic to the zero section
has a generating function quadratic at infinity,
and by Chaperon \cite{Chaperon}
and Chekanov \cite{Chekanov}
that every Legendrian submanifold of $(J^1B, \xi_{\can})$
contact isotopic to the zero section
has a generating function quadratic at infinity.
Moreover,
we have the following uniqueness theorem.
Two functions $F_1: B \times \mathbb{R}^{N_1} \rightarrow \mathbb{R}$
and $F_2: B \times \mathbb{R}^{N_2} \rightarrow \mathbb{R}$
quadratic at infinity
are said to be \emph{equivalent}
if there are non-degenerate quadratic forms $Q_1$ and $Q_2$
and a fibre preserving diffeomorphism $\Phi$
such that $F_1 \oplus Q_1 = (F_2 \oplus Q_2) \circ \Phi$.
It has been proved by Viterbo and Th\'eret \cites{Viterbo, Theret}
that any two generating functions quadratic at infinity
for the same Lagrangian submanifold of $(T^{\ast}B, \omega_{\can})$
Hamiltonian isotopic to the zero section
are equivalent up to addition of a constant,
and that any two generating functions quadratic at infinity
for the same Legendrian submanifold of $(J^1B, \xi_{\can})$
contact isotopic to the zero section
are equivalent.

Consider now the standard symplectic Euclidean space
$(\mathbb{R}^{2n}, \omega_0)$,
and the product
\[
\big( \mathbb{R}^{2n} \times \mathbb{R}^{2n} \,,\,
- \omega_0 \oplus \omega_0\big) \,.
\]
We use complex coordinates $(z_1, \cdots, z_n)$
on $\mathbb{R}^{2n}$,
where $z_j = x_j + iy_j$.
The map
\[
\tau: \mathbb{R}^{2n}\times\mathbb{R}^{2n} \rightarrow T^{\ast}\mathbb{R}^{2n} \,,\;
\tau \, (z, Z) = \Big(\frac{z+Z}{2} \,,\, i \, (Z - z) \Big)
\]
is a symplectomorphism,
and sends the diagonal to the zero section.
For any symplectomorphism $\varphi$
of $(\mathbb{R}^{2n} , \omega_0)$
we denote by
\[
\Gamma_{\varphi}:
\mathbb{R}^{2n} \rightarrow T^{\ast}\mathbb{R}^{2n}
\]
the composition of the graph
\[
\gr (\varphi):\mathbb{R}^{2n} \rightarrow
\mathbb{R}^{2n} \times \mathbb{R}^{2n} \,,\;
p \mapsto \big( p, \varphi(p) \big)
\]
with $\tau$.
We say that $F: \mathbb{R}^{2n} \times \mathbb{R}^N \rightarrow \mathbb{R}$
is a generating function for $\varphi$
if it is a generating function for the Lagrangian submanifold
$L_{\varphi} = \im (\Gamma_{\varphi})$
of $(T^{\ast}\mathbb{R}^{2n}, \omega_{\can})$.
Then $i_F: \Sigma_F \rightarrow T^{\ast}\mathbb{R}^{2n}$
induces a diffeomorphism between $\Sigma_F$
and $L_{\varphi}$,
and the composition
\begin{equation*}
\begin{tikzcd}
{\Phi_F:\,\mathbb{R}^{2n}} \arrow[r, "\Gamma_{\varphi}"] & L_{\varphi} \arrow[r] &[3ex] \Sigma_F 
\end{tikzcd}
\end{equation*}
of the inverse of this diffeomorphism with $\Gamma_{\varphi}$
induces a bijection
between the set of fixed points of $\varphi$
and the set of critical points of $F$.
Under this bijection,
a critical point $(z, \zeta)$ of $F$
corresponds to the fixed point $z$ of $\varphi$.

If $\varphi$ is a compactly supported
Hamiltonian diffeomorphism of $(\mathbb{R}^{2n} , \omega_0)$
then the Lagrangian submanifold $L_{\varphi}$
of $( T^{\ast}\mathbb{R}^{2n}, \omega_{\can})$
extends to a Lagrangian submanifold $\overline{L_{\varphi}}$
of $( T^{\ast}S^{2n}, \omega_{\can})$,
Hamiltonian isotopic to the zero section,
by seeing $S^{2n}$ as the $1$-point compactification of $\mathbb{R}^{2n}$.
We denote by $p_{\infty}$ the point at infinity of $S^{2n}$,
and identify $\mathbb{R}^{2n}$
with $S^{2n} \smallsetminus \{p_{\infty}\}$
by the stereographic projection.
We say that $\overline{F}: S^{2n} \times \mathbb{R}^N \rightarrow \mathbb{R}$
is a generating function (quadratic at infinity) for $\varphi$
if it is a generating function (quadratic at infinity)
for $\overline{L_{\varphi}}$.
If $\overline{F}: S^{2n} \times \mathbb{R}^N \rightarrow \mathbb{R}$
is a generating function quadratic at infinity for $\varphi$
then we also say that the induced function
$F$ on $\mathbb{R}^{2n} \times \mathbb{R}^N$
(which is a generating function for $L_{\varphi}$,
hence for $\varphi$)
is a generating function quadratic at infinity for $\varphi$,
and denote by $F_{\infty}$
the quadratic form on $\mathbb{R}^{2n} \times \mathbb{R}^N$
induced by $\overline{F}_{\infty}$.
By \cites{Sikorav87, Viterbo, Theret},
any compactly supported Hamiltonian diffeomorphism
of $(\mathbb{R}^{2n}, \omega_0)$
has a generating function quadratic at infinity,
unique up to equivalence and addition of a constant.

We say that a generating function
$\overline{F}: S^{2n} \times \mathbb{R}^N \rightarrow \mathbb{R}$
for a compactly supported Hamiltonian diffeomorphism $\varphi$
of $(\mathbb{R}^{2n} , \omega_0)$
is normalized if it is zero at the critical point
corresponding to the intersection of $\overline{L_{\varphi}}$ 
with the zero section at the point at infinity
$p_{\infty}$ of $S^{2n}$
(equivalently,
if the induced function $F$ on $\mathbb{R}^{2n} \times \mathbb{R}^N$
is zero at the critical points
corresponding to the fixed points of $\varphi$
outside the support).
Recall that the symplectic action
of a fixed point $p$ of a compactly supported
Hamiltonian diffeomorphism $\varphi$ of $(\mathbb{R}^{2n}, \omega_0)$
is defined by
\[
\mathcal{A}_{\varphi}(p)
= \int_0^1 \big( \lambda_0 (X_t) + H_t \big)
\big( \varphi_t(p) \big) \, dt \,,
\]
where $\lambda_0 = \sum_{j=1}^n \frac{x_j dy_j - y_j dx_j}{2}$
is the standard Liouville form,
$\{ \varphi_t \}_{t \in [0,1]}$ is any compactly supported
Hamiltonian isotopy with $\varphi_1 = \varphi$,
$X_t$ is the vector field generating this isotopy
and $H_t$ is the associated compactly supported Hamiltonian function,
with the sign convention
$\iota_{X_t}\omega_0 = dH_t$.
The action spectrum of $\varphi$ is the set $\mathcal{A}_{\varphi}$
of actions of all the fixed points of $\varphi$.
If $\overline{F}: S^{2n} \times \mathbb{R}^N \rightarrow \mathbb{R}$
is a normalized generating function quadratic at infinity for $\varphi$
then its set of critical values
coincides with $\mathcal{A}_{\varphi}$
(see for instance \cite[Proposition 2.5]{FSZ}).

Consider now the standard contact Euclidean space
$\big( \mathbb{R}^{2n+1}, \xi_0 = \ker (\alpha_0) \big)$,
and the product
\[
\big( \, \mathbb{R}^{2n+1} \times \mathbb{R}^{2n+1} \times \mathbb{R}
\,,\, \ker \, ( \pr_2^{\ast} \alpha_0
- e^{\rho} \pr_1^{\ast} \alpha_0 ) \, \big)\,,
\]
where $\rho$ denotes the coordinate in $\mathbb{R}$
and $\pr_1$ and $\pr_2$ the projections
on the first and second factors respectively.
The map
\[
\underline{\tau}:
\mathbb{R}^{2n+1} \times \mathbb{R}^{2n+1} \times \mathbb{R}
\longrightarrow J^1\mathbb{R}^{2n+1}
\]
defined by
\[
\underline{\tau} \, (z, \theta, Z, \Theta, \rho) =
\Big(\, \frac{e^{\frac{\rho}{2}} z + Z}{2} \,,\, \theta \,,\,
i \, (Z - e^{\frac{\rho}{2}} z) \,,\,
e^{\rho} - 1 \,,\,
\Theta - \theta
+ \frac{ \langle \, i \,Z \,,\, e^{\frac{\rho}{2}} z \, \rangle }{2} \,\Big) \,,
\]
where we denote by $\langle \,\cdot\, \,,\, \,\cdot\, \rangle$
the Euclidean scalar product,
is a contactomorphism,
and sends the Legendrian diagonal
\[
\Delta = \{\, (z, \theta, z, \theta, 0) \,\}
\subset \mathbb{R}^{2n+1} \times \mathbb{R}^{2n+1} \times \mathbb{R}
\]
to the zero section.
This contactomorphism is strict
with respect to the contact forms
that we are considering:
the pullback by $\underline{\tau}$
of the canonical contact form $dw - \lambda_{\can}$
on $J^1\mathbb{R}^{2n+1}$
is equal to the contact form
$\pr_2^{\ast}\alpha_0 - e^{\rho} \pr_1^{\ast} \alpha_0$
on $\mathbb{R}^{2n+1} \times \mathbb{R}^{2n+1} \times \mathbb{R}$.
For any contactomorphism $\phi$
of $(\mathbb{R}^{2n+1}, \xi_0)$
we denote by
\[
\Gamma_{\phi}: \mathbb{R}^{2n+1} \rightarrow
J^1 \mathbb{R}^{2n+1}
\]
the composition of the Legendrian graph
\[
\gr(\phi): \mathbb{R}^{2n+1} \rightarrow
\mathbb{R}^{2n+1} \times \mathbb{R}^{2n+1} \times \mathbb{R} \,,\;
p \mapsto \big( p, \phi(p), g(p) \big) \,,
\]
where $g$ denotes the conformal factor of $\phi$
with respect to $\alpha_0$,
with $\underline{\tau}$.
We say that
$F: \mathbb{R}^{2n+1} \times \mathbb{R}^N \rightarrow \mathbb{R}$
is a generating function for $\phi$
if it is a generating function
for the Legendrian submanifold $\Lambda_{\phi} = \im (\Gamma_{\phi})$
of $( J^1\mathbb{R}^{2n+1}, \xi_{\can} )$.
Then $j_F: \Sigma_F \rightarrow J^1\mathbb{R}^{2n+1}$
induces a diffeomorphism between $\Sigma_F$
and $\Lambda_{\phi}$,
and the composition
\begin{equation}\label{equation: diffeomorphism gf contact}
\begin{tikzcd}
{\Psi_F:\,\mathbb{R}^{2n+1}} \arrow[r, "\Gamma_{\phi}"] & \Lambda_{\phi} \arrow[r] &[3ex] \Sigma_F 
\end{tikzcd}
\end{equation}
of the inverse of this diffeomorphism with $\Gamma_{\phi}$
induces a bijection
between the set of translated points of $\phi$
and the set of critical points of $F$.
Under this bijection,
a critical point $(z, \theta, \zeta)$ of $F$
corresponds to the translated point $(z, \theta)$ of $\phi$.
The critical value $F (z, \theta, \zeta)$
is equal to the action
of the translated point $(z, \theta)$,
i.e.\ the length of the Reeb chord
from $(z, \theta)$ to $\phi (z, \theta)$.

The next lemma,
which follows directly from the definitions,
will be used later.

\begin{lemma}\label{lemma: useful}
Let $F: \mathbb{R}^{2n+1} \times \mathbb{R}^N \rightarrow \mathbb{R}$
be a generating function for a contactomorphism
$\phi$ of $(\mathbb{R}^{2n+1}, \xi_0)$.
Let $(z, \theta, \zeta)$ be a fibre critical point of $F$,
and denote by $(z', \theta') 
\in \mathbb{R}^{2n+1}$
the image of $(z, \theta, \zeta)$
by the inverse of the diffeomorphism
\eqref{equation: diffeomorphism gf contact}.
Denote by $\phi_z$ and $\phi_{\theta}$
the components of $\phi$,
and by $g$ the conformal factor.
Then
$$
\begin{cases}
z =
\frac{e^{\frac{1}{2} g (z', \theta')} \, z'
\,+\, \phi_z (z', \theta')}{2} \\
\theta = \theta' \,,
\end{cases}
$$
\begin{equation*}
\begin{cases}
\frac{\partial F}{\partial z} \, (z, \theta, \zeta)
= i \; \big( \phi_z (z', \theta') 
- e^{\frac{1}{2} g (z', \theta')} \, z' \big) \\
\frac{\partial F}{\partial \theta} \, (z, \theta, \zeta)
= e^{g (z', \theta')} - 1 
\end{cases}    
\end{equation*}
and
$$
F (z, \theta, \zeta)
= \phi_{\theta} (z', \theta') - \theta'
+ \frac{1}{2} \; e^{\frac{1}{2} g (z', \theta')}
\; \big\langle\, i \, \phi_z (z', \theta') \,,\, z' \,\big\rangle \,.
$$
\end{lemma}

Let $\phi$ be a compactly supported contactomorphism
of $(\mathbb{R}^{2n} \times S^1, \xi_0)$
contact isotopic to the identity,
and denote by $\phi_{\mathbb{R}}$ its lift
to $(\mathbb{R}^{2n+1}, \xi_0)$,
i.e.\ the contactomorphism of $(\mathbb{R}^{2n+1}, \xi_0)$
that projects to $\phi$ and is the identity
over the complement of the support of $\phi$.
We say that
$F: \mathbb{R}^{2n+1} \times \mathbb{R}^N \rightarrow \mathbb{R}$
is a generating function for $\phi$
if it is a generating function
for the Legendrian submanifold $\Lambda_{\phi_{\mathbb{R}}}$
of $(J^1 \mathbb{R}^{2n+1}, \xi_0)$.
The latter extends to a Legendrian submanifold 
of $\big( J^1 (S^{2n} \times \mathbb{R}) , \xi_{\can} \big)$,
which descends to a Legendrian submanifold $\overline{\Lambda_{\phi}}$
of $\big( J^1 (S^{2n} \times S^1) , \xi_{\can} \big)$,
contact isotopic to the zero section.
We say that
$\overline{F}: (S^{2n} \times S^1) \times \mathbb{R}^N \rightarrow \mathbb{R}$
is a generating function (quadratic at infinity) for $\phi$
if it is a generating function (quadratic at infinity)
for $\overline{\Lambda_{\phi}}$.
If $\overline{F}:
(S^{2n} \times S^1) \times \mathbb{R}^N \rightarrow \mathbb{R}$
is a generating function quadratic at infinity for $\phi$
then we also say that the induced function $F$
on $\mathbb{R}^{2n + 1} \times \mathbb{R}^N$
(which is a generating function for $\Lambda_{\phi_{\mathbb{R}}}$,
hence for $\phi$)
is a generating function quadratic at infinity for $\phi$,
and denote by $F_{\infty}$ the quadratic form
on $\mathbb{R}^{2n + 1} \times \mathbb{R}^N$
induced by $\overline{F}_{\infty}$.
The function $F$ is invariant by the action of $\mathbb{Z}$
on $\mathbb{R}^{2n + 1} \times \mathbb{R}^N$
generated by the map
\begin{equation}\label{equation: action Z for F}
(z, \theta, \zeta)
\mapsto (z, \theta + 1, \zeta) \,.
\end{equation}
By \cites{Chaperon, Chekanov, Viterbo, Theret},
any compactly supported contactomorphism of $(\mathbb{R}^{2n} \times S^1, \xi_0)$
contact isotopic to the identity
has a generating function quadratic at infinity,
unique up to equivalence.

The action of a translated point $p$
of a compactly supported contactomorphism $\phi$
of $(\mathbb{R}^{2n} \times S^1, \xi_0)$
is defined to be the action
of any point of $\mathbb{R}^{2n+1}$
projecting to $p$,
seen as a translated point of $\phi_{\mathbb{R}}$.
The action spectrum of $\phi$
is the set $\mathcal{A}_{\phi} := \mathcal{A}_{\phi_{\mathbb{R}}}$
of actions of all the translated points of $\phi$.
If $\overline{F}: S^{2n} \times S^1 \times \mathbb{R}^N \rightarrow \mathbb{R}$
is a generating function quadratic at infinity for $\phi$
then its set of critical values
coincides with $\mathcal{A}_{\phi}$.

We say that a generating function quadratic at infinity $F$
for a compactly supported Hamiltonian diffeomorphism
of $(\mathbb{R}^{2n}, \omega_0)$
or for a compactly supported contactomorphism
of $(\mathbb{R}^{2n} \times S^1, \xi_0)$
contact isotopic to the identity
is \emph{special}
if the quadratic form $F_{\infty}$ does not depend
on the base variable
and $F = F_{\infty}$ outside a compact set.
In the case of Hamiltonian diffeomorphisms,
such $F$ is then normalized.
By \cite[Lemma 2.10]{FSZ},
all compactly supported Hamiltonian diffeomorphisms
of $(\mathbb{R}^{2n}, \omega_0)$
and all compactly supported contactomorphisms
of $(\mathbb{R}^{2n} \times S^1, \xi_0)$
contact isotopic to the identity
have a special generating function quadratic at infinity.
Moreover,
we have the following result,
proved in \cite[Lemma 2.11]{FSZ}.
Recall from \cites{Viterbo, San11} that the groups
of compactly supported Hamiltonian diffeomorphisms
of $(\mathbb{R}^{2n}, \omega_0)$
and of compactly supported contactomorphisms
of $(\mathbb{R}^{2n} \times S^1, \xi_0)$
contact isotopic to the identity
admit a partial order,
defined by posing $\phi \leq \psi$
if $\psi \circ \phi^{-1}$ is the time-$1$ map
of the flow of a non-negative compactly supported Hamiltonian function.

\begin{prop}\label{proposition: monotonicity}
If $\phi$ and $\psi$ are compactly supported
Hamiltonian diffeomorphisms of $(\mathbb{R}^{2n}, \omega_0)$
or compactly supported contactomorphisms
of $(\mathbb{R}^{2n} \times S^1, \xi_0)$
contact isotopic to the identity
with $\phi \leq \psi$
then there are special generating functions quadratic at infinity
$F$ and $G$ for $\phi$ and $\psi$ respectively
that are defined on the same domain
and satisfy $F \leq G$.
\end{prop}

Recall that the lift to $(\mathbb{R}^{2n} \times S^1, \xi_0)$
of a compactly supported Hamiltonian diffeomorphism
$\varphi$ of $(\mathbb{R}^{2n}, \omega_0)$
is the (compactly supported) contactomorphism
$\widetilde{\varphi}$ defined by
\[
\widetilde{\varphi} \, (z, \theta)
= \big( \varphi(z) \,,\, \theta + S (z) \big) \,,
\]
where $S$ is the compactly supported function
that satisfies $\varphi^{\ast} \lambda_0 - \lambda_0 = dS$.
The following result
is proved in \cite[Lemma 3.2]{San11}
or \cite[Lemma 2.9]{FSZ}.

\begin{lemma}\label{lemma: gf lift to contact}
If $f: \mathbb{R}^{2n} \times \mathbb{R}^N \rightarrow \mathbb{R}$
is a special generating function quadratic at infinity
for a compactly supported Hamiltonian diffeomorphism $\varphi$
of $(\mathbb{R}^{2n}, \omega_0)$
then
\[
F: \mathbb{R}^{2n+1} \times \mathbb{R}^N
\rightarrow \mathbb{R} \,,\;
F (z, \theta, \zeta)
= f (z, \zeta)
\]
is a special generating function quadratic at infinity
for the lift of $\varphi$ to $(\mathbb{R}^{2n} \times S^1, \xi_0)$.
\end{lemma}

We now discuss how
for any odd positive integer $k$
and any sequence $(\varphi_0, \cdots, \varphi_{k-1})$
of compactly supported Hamiltonian diffeomorphisms
of $(\mathbb{R}^{2n}, \omega_0)$
we can obtain a generating function quadratic at infinity
for the composition $\varphi_{k-1} \circ \cdots \circ \varphi_0$
starting from special generating functions quadratic at infinity
$F_j: \mathbb{R}^{2n} \times \mathbb{R}^{N_j} \rightarrow \mathbb{R}$
for the $\varphi_j$'s.
Pose $N = N_0 + \cdots + N_{k-1}$,
and observe first that the function
\[
F_0 \,\sharp\,\cdots \,\sharp\, F_{k-1}:
\mathbb{R}^{2n}
\times (\mathbb{R}^{2n(k-1)} \times \mathbb{R}^N)
\rightarrow \mathbb{R}
\]
defined by
\[
F_0 \,\sharp\,\cdots \,\sharp\, F_{k-1} \;
(z_0; z_1, \cdots, z_{k-1}, \zeta_0, \cdots, \zeta_{k-1})
= \sum_{j = 0}^{k-1}
\; F_j \Big( \frac{z_j + z_{j+1}}{2} \,,\, \zeta_j \Big)
+ \frac{1}{2} \, \langle\, i \, z_j \,,\, z_{j+1} \,\rangle
\]
is a normalized generating function
for the composition $\varphi_{k-1} \circ \cdots \circ \varphi_0$.
This is proved in \cite[Proposition 5.1]{Allais}
for generating functions without fibre variables,
and in \cite[Proposition 3.1]{FSZ}
for a composition of the form
$F^{\sharp k} = F \,\sharp \,\cdots \,\sharp\, F$.
The proof of \cite[Proposition 3.1]{FSZ}
of course can easily be adapted to the case
of a general composition $F_0 \,\sharp\,\cdots \,\sharp\, F_{k-1}$.

Consider now the fibre preserving change of variables $A$
defined by
\[
A \, (z_0; z_1, \cdots, z_{k-1}, \zeta_0, \cdots, \zeta_{k-1})
= \Big(\, z_0 \,;\,
z_1 - z_0 \,,\, \cdots \,,\, z_{k-1} - z_0 \,,\,
\zeta_0, \cdots, \zeta_{k-1} \,\Big) \,.
\]
Since $A$ is fibre preserving,
the function
$$
(F_0 \,\sharp\,\cdots \,\sharp\, F_{k-1}) \circ A^{-1}:
\mathbb{R}^{2n} \times (\mathbb{R}^{2n(k-1)}
\times \mathbb{R}^N)
\rightarrow \mathbb{R}
$$
is also a normalized generating function
for $\varphi_{k-1} \circ \cdots \circ \varphi_0$.

\begin{rmk}\label{remark: compactification}
We use here a change of variables
different from the one used in \cite{FSZ}.
Indeed,
in \cite{FSZ} the authors consider a single function $F$
and a change of variables
that behaves in a certain way with respect to the $\mathbb{Z}_k$-action
on the domain of $F^{\sharp k}$
that cyclically permutes the coordinates
(the conjugation of this $\mathbb{Z}_k$-action
by the change of variables is required to be fibre preserving).
This property is not needed in the present work
(since we work with usual homology,
and not with $\mathbb{Z}_k$-equivariant homology).
On the other hand,
the advantage of our change of variables
is that it is itself fibre preserving.
\end{rmk}

We have
\[
(F_0 \,\sharp\,\cdots \,\sharp\, F_{k-1}) \circ A^{-1} \;
(z_0; z_1, \cdots, z_{k-1}, \zeta_0, \cdots, \zeta_{k-1})
\]
\[
= F_0 \big( z_0 + \frac{z_1}{2} \,,\, \zeta_0 \big)
+ \sum_{j = 1}^{k - 2}
F_j \big( z_0 + \frac{z_j + z_{j+1}}{2} \,,\, \zeta_j \big)
+ F_{k-1} \big( z_0 + \frac{z_{k-1}}{2} \,,\, \zeta_{k-1} \big)
\]
\[
+ \sum_{j = 1}^{k - 2} \frac{1}{2} \, \langle\, i \, z_j , z_{j+1} \,\rangle \,,
\]
thus $(F_0 \,\sharp\,\cdots \,\sharp\, F_{k-1}) \circ A^{-1}$
can be extended to a function
\[
\overline{F_0 \,\sharp\,\cdots \,\sharp\, F_{k-1}}:
S^{2n} \times (\mathbb{R}^{2n(k-1)} \times \mathbb{R}^N)
\rightarrow \mathbb{R}
\]
by setting
\[
\overline{F_0 \,\sharp\,\cdots \,\sharp\, F_{k-1}} \;
(p_{\infty}; z_1, \cdots, z_{k-1}, \zeta_0, \cdots, \zeta_{k-1})
= \sum_{j = 0}^{k - 1} \overline{F_j} \, (p_{\infty}, \zeta_j)
+ \sum_{j = 1}^{k - 2} \frac{1}{2} \, \langle\, i \, z_j , z_{j+1} \,\rangle \,.
\]

\begin{prop}\label{proposition: gfqi k}
For any odd positive integer $k$,
sequence $(\varphi_0, \cdots, \varphi_{k-1})$
of compactly supported Hamiltonian diffeomorphisms
of $(\mathbb{R}^{2n}, \omega_0)$
and special generating functions quadratic at infinity
$F_j$ for the $\varphi_j$'s,
the function $\overline{F_0 \,\sharp\,\cdots \,\sharp\, F_{k-1}}$
is a normalized generating function quadratic at infinity
for $\varphi_{k-1} \circ \cdots \circ \varphi_0$.
\end{prop}

\begin{proof}
Since $(F_0 \,\sharp\,\cdots \,\sharp\, F_{k-1}) \circ A^{-1}$
is a generating function for $\varphi := \varphi_{k-1} \circ \cdots \circ \varphi_0$,
hence for the Lagrangian submanifold
$L_{\varphi}$ of $T^{\ast}\mathbb{R}^{2n}$,
the compactification $\overline{F_0 \,\sharp\,\cdots \,\sharp\, F_{k-1}}$
is a generating function for the Lagrangian submanifold
$\overline{L_{\varphi}}$ of $T^{\ast}S^{2n}$.
Consider the non-degenerate quadratic form $Q$
on $S^{2n} \times (\mathbb{R}^{2n(k-1)} \times \mathbb{R}^N)$
defined by
\[
Q (p; z_1, \cdots, z_{k-1}, \zeta_0, \cdots, \zeta_{k-1})
= \sum_{j = 0}^{k-1} (F_j)_{\infty} (\zeta_j)
+ \sum_{j = 1}^{k - 2} \frac{1}{2} \, \langle\, i \, z_j , z_{j+1} \,\rangle \,.
\]
Since $\lVert \frac{\partial F_j}{\partial z_j} \rVert$
and $\lVert \frac{\partial}{\partial \zeta_j} \big(F_j - (F_j)_{\infty}\big) \rVert$
are bounded,
$\partial_v ( \overline{F_0 \,\sharp\,\cdots \,\sharp\, F_{k-1}} - Q )$ is bounded,
and thus $\overline{F_0 \,\sharp\,\cdots \,\sharp\, F_{k-1}}$
is a generating function quadratic at infinity
for $\overline{L_{\varphi}}$,
hence for $\varphi$.
Since the generating function
$(F_0 \,\sharp\,\cdots \,\sharp\, F_{k-1}) \circ A^{-1}$ is normalized,
$\overline{F_0 \,\sharp\,\cdots \,\sharp\, F_{k-1}}$ is also normalized.
\end{proof}

Proposition \ref{proposition: gfqi k} will be used
in Section \ref{section: spectral selectors}
to compare the spectral selectors
associated to a sequence of contactomorphisms
of $(\mathbb{R}^{2n} \times S^1, \xi_0)$
of the form
$(\widetilde{\varphi_0}, \cdots, \widetilde{\varphi_{k-1}})$
to the spectral selectors
of the composition $\varphi_{k-1} \circ \cdots \circ \varphi_0$.


\section{Functions quadratic at infinity
detecting flat translated chains of sequences of contactomorphisms
of $(\mathbb{R}^{2n} \times S^1, \xi_0)$}\label{section: functions translated chains}

We have seen in the previous section
how to obtain generating functions quadratic at infinity
for odd compositions of compactly supported
Hamiltonian diffeomorphisms of $(\mathbb{R}^{2n}, \omega_0)$.
We now describe
a similar construction
to obtain functions quadratic at infinity
detecting flat translated chains
of a sequence $\bm{\phi} = (\phi_0, \cdots, \phi_{k-1})$
of compactly supported contactomorphisms
of $(\mathbb{R}^{2n} \times S^1, \xi_0)$
contact isotopic to the identity.

Let $F_j: \mathbb{R}^{2n+1} \times \mathbb{R}^{N_j} \rightarrow \mathbb{R}$
be special generating functions quadratic at infinity
for the $\phi_j$'s.
Pose $N = N_0 + \cdots + N_{k-1}$,
and consider the function
\[
F_0 \,\sharp\, \cdots \,\sharp\, F_{k-1}:
\mathbb{R}^{(2n+1)k} \times \mathbb{R}^N 
\rightarrow \mathbb{R}
\]
defined by
\begin{gather*}
F_0 \,\sharp\, \cdots \,\sharp\, F_{k-1}
\; (z_0, \theta_0, \cdots, z_{k-1}, \theta_{k-1},
\zeta_0, \cdots, \zeta_{k-1}) \\
= \sum_{j = 0}^{k-1}
F_j \Big( \frac{z_j + z_{j+1}}{2} \,,\, \theta_j \,,\, \zeta_j \Big)
+ \frac{1}{2} \, \langle\, i \, z_j , z_{j+1} \,\rangle \,.
\end{gather*}

\begin{prop}\label{proposition: composition formula contact flat}
The map
\[
(z_0, \theta_0, \cdots, z_{k-1}, \theta_{k-1}, \zeta_0, \cdots, \zeta_{k-1})
\mapsto \big( (z_0, \theta_0), \cdots, (z_{k-1}, \theta_{k-1}) \big)
\]
is a bijection
between the set of critical points
of $F_0 \,\sharp\, \cdots \,\sharp\, F_{k-1}$
and the set of flat translated chains of $\bm{\phi}_{\mathbb{R}}$.
Moreover,
the critical values of $F_0 \,\sharp\, \cdots \,\sharp\, F_{k-1}$
are given by the actions of the corresponding translated chains.
\end{prop}

\begin{proof}
Let $p = (z_0, \theta_0, \cdots, z_{k-1}, \theta_{k-1}, \zeta_0, \cdots, \zeta_{k-1})$
be a critical point of $F := F_0 \,\sharp\, \cdots \,\sharp\, F_{k-1}$.
For every $j$ we have
\[
0 = \frac{\partial F}{\partial \zeta_j} \; (p)
= \frac{\partial F_j}{\partial \zeta_j}
\; \Big( \frac{z_j + z_{j+1}}{2} \,,\, \theta_j \,,\, \zeta_j \Big) \,,
\]
thus $\big( \frac{z_j + z_{j+1}}{2} \,,\, \theta_j \,,\, \zeta_j \big)$
is a fibre critical point of $F_j$.
Pose
\begin{equation}\label{equation: definition XY}
\begin{cases}
Z_j = \frac{z_j + z_{j+1}}{2} \\
\Theta_j = \theta_j \,.
\end{cases}    
\end{equation}
Since $k$ is odd,
the $z_j$'s are determined by the $Z_j$'s
via the formula
\[
z_j = \sum_{l=0}^{k-1} (-1)^l \, Z_{j+l} \,.
\]
We also have
\begin{equation}\label{equation: k odd}
\frac{1}{2} \, (z_{j-1} - z_{j+1}) = Z_{j-1} - Z_j \,.   
\end{equation}
Let $(Z_j', \Theta'_j) \in \mathbb{R}^{2n+1}$
be the image of $(Z_j, \Theta_j, \zeta_j)$
by the inverse of the diffeomorphism
\eqref{equation: diffeomorphism gf contact}.
By Lemma \ref{lemma: useful},
we have
\begin{equation}\label{equation: conditions XY}
\begin{cases}
Z_j = \frac{e^{\frac{1}{2} (g_j)_{\mathbb{R}} (Z'_j, \Theta'_j)} \, Z'_j
\,+\, ((\phi_j)_{\mathbb{R}})_z \, (Z'_j, \Theta'_j)}{2} \\
\Theta_j = \Theta'_j \,,
\end{cases}
\end{equation}
\begin{equation*}
\begin{cases}
\frac{\partial F_j}{\partial Z_j} \, (Z_j, \Theta_j, \zeta_j)
= i \; \big(\, ((\phi_j)_{\mathbb{R}})_z \, (Z'_j, \Theta'_j) 
- e^{\frac{1}{2} (g_j)_{\mathbb{R}} (Z'_j, \Theta'_j)} \, Z'_j \,\big) \\
\frac{\partial F_j}{\partial \Theta_j} \, (Z_j, \Theta_j, \zeta_j)
= e^{(g_j)_{\mathbb{R}} (Z'_j, \Theta'_j)} - 1 
\end{cases}    
\end{equation*}
and
\begin{equation}\label{equation: value F}
F_j (Z_j, \Theta_j, \zeta_j)
= ((\phi_j)_{\mathbb{R}})_{\theta} \, (Z'_j, \Theta'_j)
- \Theta'_j
+ \frac{1}{2} \; e^{\frac{1}{2} (g_j)_{\mathbb{R}} (Z'_j, \Theta'_j)}
\; \big\langle\, i \, ((\phi_j)_{\mathbb{R}})_z \, (Z'_j, \Theta'_j)
\,,\, Z'_j \,\big\rangle \,,
\end{equation}
where $((\phi_j)_{\mathbb{R}})_z$ and $((\phi_j)_{\mathbb{R}})_{\theta}$
are the components of $(\phi_j)_{\mathbb{R}}$,
and $(g_j)_{\mathbb{R}}$ the conformal factor.
For every $j$ we have
\[
0 = \frac{\partial F}{\partial \theta_j} \, (p)
= \frac{\partial F_j}{\partial \Theta_j} \, (Z_j, \Theta_j, \zeta_j)
= e^{(g_j)_{\mathbb{R}} (Z'_j, \Theta'_j)} - 1 \,,
\]
thus
\begin{equation}\label{equation: condition derivatives Theta}
(g_j)_{\mathbb{R}} (Z'_j, \Theta'_j) = 0 \,.
\end{equation}
Moreover,
using \eqref{equation: k odd}
and \eqref{equation: condition derivatives Theta}
we have
\[
0 = \; \frac{\partial F}{\partial z_j} \, (p)
= \; \frac{1}{2} \, 
\frac{\partial F_j}{\partial Z_j} \, (Z_j, \Theta_j, \zeta_j)
+ \frac{1}{2} \, \frac{\partial F_{j-1}}{\partial Z_{j-1}}
\, (Z_{j-1}, \Theta_{j-1}, \zeta_{j-1})
+ \frac{1}{2} \, i \, (z_{j-1} - z_{j+1})
\]
\[
=  \; \frac{1}{2} \,
\frac{\partial F_j}{\partial Z_j} \, (Z_j, \Theta_j, \zeta_j)
+ \frac{1}{2} \, \frac{\partial F_{j-1}}{\partial Z_{j-1}} \,(Z_{j-1}, \Theta_{j-1}, \zeta_{j-1})
+ i \, (Z_{j-1} - Z_j)
\]
\[
= \; i \; \Big( ((\phi_{j-1})_{\mathbb{R}})_z \, (Z'_{j-1}, \Theta'_{j-1})
- Z'_j \Big) \,,
\]
thus
\begin{equation}\label{equation: condition derivatives x}
Z'_j
= ((\phi_{j-1})_{\mathbb{R}})_z (Z'_{j-1}, \Theta'_{j-1}) \,.
\end{equation}
We thus conclude that
$\big( (Z'_0, \Theta'_0) \,,\, \cdots \,,\,
(Z'_{k-1}, \Theta'_{k-1}) \big)$
is a flat translated chain of $\bm{\phi}_{\mathbb{R}}$.

Conversely,
let $(p_0, \cdots, p_{k-1})$ be a flat translated chain
of $\bm{\phi}_{\mathbb{R}}$.
For every $j$,
denote by $(Z_j, \Theta_j, \zeta_j) \in \Sigma_F$
the image of $p_j$
by the diffeomorphism \eqref{equation: diffeomorphism gf contact},
and define $(z_j, \theta_j)$
by the relation \eqref{equation: definition XY}
(using that $k$ is odd).
Calculations as above show that
$(z_0, \theta_0, \cdots, z_{k-1}, \theta_{k-1},
\zeta_0, \cdots, \zeta_{k-1})$
is then a critical point of $F$.
We thus deduce that the map
\begin{equation}\label{equation: bijection}
(z_0, \theta_0, \cdots, z_{k-1}, \theta_{k-1}, \zeta_0, \cdots, \zeta_{k-1})
\mapsto 
\big( (Z'_0, \Theta'_0) \,,\, \cdots \,,\,
(Z'_{k-1}, \Theta'_{k-1}) \big)
\end{equation}
is a bijection between the set of critical points of $F$
and the set of flat translated chains of $\bm{\phi}_{\mathbb{R}}$.

Observe now that the relations
\eqref{equation: conditions XY},
\eqref{equation: condition derivatives Theta}
and \eqref{equation: condition derivatives x}
give
\begin{equation}\label{equation: XY bis}
\begin{cases}
Z_j = \frac{Z'_j + Z'_{j+1}}{2} \\
\Theta_j = \Theta'_j \,.
\end{cases}    
\end{equation}
Since $k$ is odd,
\eqref{equation: XY bis} and \eqref{equation: definition XY}
imply that $(Z'_j, \Theta'_j) = (z_j, \theta_j)$.
We thus conclude that the bijection \eqref{equation: bijection}
is simply given by
\[
(z_0, \theta_0, \cdots, z_{k-1}, \theta_{k-1}, \zeta_0, \cdots, \zeta_{k-1})
\mapsto \big( (z_0, \theta_0) \,,\, \cdots \,,\, (z_{k-1}, \theta_{k-1}) \big) \,.
\]
Finally,
using \eqref{equation: value F} we see that
for a critical point
$p = (z_0, \theta_0, \cdots, z_{k-1}, \theta_{k-1}, \zeta_0, \cdots, \zeta_{k-1})$ of $F$
we have
\[
F (p)
= \sum_{j = 0}^{k-1} F_j (Z_j, \Theta_j, \zeta_j)
+ \frac{1}{2} \, \langle\, i \, z_j , z_{j+1} \,\rangle 
\]
\[
= \sum_{j = 0}^{k-1}
\, ((\phi_j)_{\mathbb{R}})_{\theta} \, (Z'_j, \Theta'_j) - \Theta'_j
+ \frac{1}{2} \; \big\langle\,
i \, ((\phi_j)_{\mathbb{R}})_z \, (Z'_j, \Theta'_j)
\,,\, Z'_j \,\big\rangle
+ \frac{1}{2} \, \langle\, i \, z_j , z_{j+1}  \,\rangle
\]
\[
= \sum_{j = 0}^{k-1}
\,((\phi_j)_{\mathbb{R}})_{\theta}\, (Z'_j, \Theta'_j) - \Theta'_j
+ \frac{1}{2} \; \big\langle\, i \, Z'_{j+1} \,,\, Z'_j \,\big\rangle
+ \frac{1}{2} \, \langle\, i \, z_j , z_{j+1} \,\rangle
\]
\[
= \sum_{j = 0}^{k-1}
((\phi_j)_{\mathbb{R}})_{\theta} \, (Z'_j, \Theta'_j)
- \Theta'_j \,,
\]
which is the action of the translated chain
\[
\big( (Z'_0, \Theta'_0) \,,\, \cdots \,,\, (Z'_{k-1}, \Theta'_{k-1}) \big)
= \big( (z_0, \theta_0) \,,\, \cdots \,,\, (z_{k-1}, \theta_{k-1}) \big) \,.
\qedhere
\]
\end{proof}

We now compactify $F_0 \,\sharp\, \cdots \,\sharp\, F_{k-1}$
to a function quadratic at infinity.
We consider the change of variables $B$
defined by
\[
B \, (z_0, \theta_0;
z_1, \theta_1, \cdots, z_{k-1}, \theta_{k-1}, \zeta_0, \cdots, \zeta_{k-1})
\]
\[
= \Big(\, z_0, \theta_0 \,;\,
z_1 - z_0 \,,\, \theta_1 \,,\, \cdots \,,\,
z_{k-1} - z_0 \,,\, \theta_{k-1} \,,\,
\zeta_0, \cdots, \zeta_{k-1} \,\Big) \,.
\]

We have
\[
(F_0 \,\sharp\, \cdots \,\sharp\, F_{k-1}) \circ B^{-1} \;
(z_0, \theta_0, \cdots, z_{k-1}, \theta_{k-1}, \zeta_0, \cdots, \zeta_{k-1})
\]
\[
= F_0 \Big( z_0 + \frac{z_1}{2} \,,\, \theta_0 \,,\, \zeta_0 \Big)
+ \sum_{j = 1}^{k - 2}
F_j \Big( z_0 + \frac{z_j + z_{j+1}}{2} \,,\, \theta_j \,,\, \zeta_j \Big)
+ F_{k-1} \Big( z_0 + \frac{z_{k-1}}{2} \,,\, \theta_{k-1} \,,\, \zeta_{k-1} \Big)
\]
\[
+ \sum_{j = 1}^{k - 2} \frac{1}{2} \, \langle\, i \, z_j , z_{j+1} \,\rangle \,,
\]
thus $(F_0 \,\sharp\, \cdots \,\sharp\, F_{k-1}) \circ B^{-1}$
can be extended to a function
on $S^{2n} \times \mathbb{R} \times \mathbb{R}^{(2n+1)(k-1)}
\times \mathbb{R}^N$
by setting the value
at $(p_{\infty}, \theta_0, z_1, \theta_1, \cdots, z_{k-1}, \theta_{k-1},
\zeta_0, \cdots, \zeta_{k-1})$
to be
\[
\sum_{j = 0}^{k-1} \overline{F}_j \, (p_{\infty}, \theta_j, \zeta_j)
+ \sum_{j=1}^{k-2} \frac{1}{2} \, \langle\, i \, z_j , z_{j+1} \,\rangle \,.
\]
Since the $F_j$'s are invariant by the $\mathbb{Z}$-action
\eqref{equation: action Z for F},
this latter function descends to a function
\[
\overline{F_0 \,\sharp\, \cdots \,\sharp\, F_{k-1}}:
S^{2n} \times S^1 \times (\mathbb{R}^{2n} \times S^1)^{k-1}
\times \mathbb{R}^N
\]
\[
\equiv S^{2n} \times \mathbb{T}^k \times \mathbb{R}^{2n(k-1)}
\times \mathbb{R}^N
\rightarrow \mathbb{R} \,.
\]

\begin{prop}\label{proposition: quadratic at infinity flat chains}
For any odd positive integer $k$,
sequence $\bm{\phi} = (\phi_0, \cdots, \phi_{k-1})$
of compactly supported contactomorphisms
of $(\mathbb{R}^{2n} \times S^1, \xi_0)$
contact isotopic to the identity
and special generating functions quadratic at infinity $F_j$
for the $\phi_j$'s,
the function $\overline{F_0 \,\sharp\, \cdots \,\sharp\, F_{k-1}}$
is quadratic at infinity over $S^{2n} \times \mathbb{T}^k$.
Moreover,
its set of critical values is equal to $\mathcal{A}_{\bm{\phi}}$.
\end{prop}

\begin{proof}
Consider the non-degenerate quadratic form $Q$
on $(S^{2n} \times \mathbb{T}^k) \times \mathbb{R}^{2n(k-1)}
\times \mathbb{R}^N$
defined by
\[
Q \big(p, (\theta_0, \cdots, \theta_{k-1});
z_1, \cdots, z_{k-1}, \zeta_0, \cdots, \zeta_{k-1} \big)
= \sum_{j = 0}^{k-1} \, (F_j)_{\infty} (\zeta_j)
+ \sum_{j = 1}^{k-2} \, \frac{1}{2} \, \langle\, i \, z_j , z_{j+1} \,\rangle \,.
\]
Since $\lVert \frac{\partial F}{\partial z_j} \rVert$
and $\lVert \frac{\partial}{\partial \zeta_j} \big(F_j - (F_j)_{\infty}\big) \rVert$
are bounded,
$\partial_v ( \overline{F_0 \,\sharp\, \cdots \,\sharp\, F_{k-1}} - Q )$ is bounded,
and thus $\overline{F_0 \,\sharp\, \cdots \,\sharp\, F_{k-1}}$ is quadratic at infinity.
The second statement follows from
Proposition \ref{proposition: composition formula contact flat}.
\end{proof}

We end this section with the following result.

\begin{prop}\label{proposition: stabilization etc contact}
Consider a sequence $\bm{\phi} = (\phi_0, \cdots, \phi_{k-1})$,
with $k$ odd,
of compactly supported contactomorphisms
of $(\mathbb{R}^{2n} \times S^1, \xi_0)$
contact isotopic to the identity,
and let $F_j^{(1)}: \mathbb{R}^{2n+1} \times \mathbb{R}^{N_j^{(1)}} \rightarrow \mathbb{R}$
and $F_j^{(2)}: \mathbb{R}^{2n+1} \times \mathbb{R}^{N_j^{(2)}} \rightarrow \mathbb{R}$
be special generating functions quadratic at infinity
for the $\phi_j$'s.
Then $\overline{F_0^{(1)} \,\sharp\, \cdots \,\sharp\, F_{k-1}^{(1)}}$
and $\overline{F_0^{(2)} \,\sharp\, \cdots \,\sharp\, F_{k-1}^{(2)}}$
are equivalent.
\end{prop}

\begin{proof}
By \cite{Viterbo, Theret},
for every $j$ the functions $\overline{F_j^{(1)}}$
and $\overline{F_j^{(2)}}$ are equivalent,
thus there are non-degenerate quadratic forms
$Q_j^{(1)}$ on $\mathbb{R}^{2n+1} \times \mathbb{R}^{M_j^{(1)}}$
and $Q_j^{(2)}$ on $\mathbb{R}^{2n+1} \times \mathbb{R}^{M_j^{(2)}}$,
for $M_j^{(1)}$ and $M_j^{(2)}$ such that
$$
N_j^{(1)} + M_j^{(1)} = N_j^{(2)} + M_j^{(2)} =: N_j \,,
$$
and a fibre preserving diffeomorphism $\Phi_j$
of $\mathbb{R}^{2n+1} \times \mathbb{R}^{N_j}$
such that
\[
F_j^{(1)} \oplus Q_j^{(1)} = (F_j^{(2)} \oplus Q_j^{(2)}) \circ \Phi_j \,.
\]
For $l = 1, 2$ we have
\[
\overline{ (F_0^{(l)} \oplus Q_0^{(l)}) \,\sharp\, \cdots \,\sharp\, (F_{k-1}^{(l)} \oplus Q_{k-1}^{(l)}) }
= \overline{ F_0^{(l)} \,\sharp\, \cdots \,\sharp\, F_{k-1}^{(l)} }
\oplus \big(Q_0^{(l)} \oplus \cdots \oplus Q_{k-1}^{(l)}\big) \,.
\]
Moreover,
\[
\overline{ (F_0^{(1)} \oplus Q_0^{(1)}) \,\sharp\, \cdots \,\sharp\, (F_{k-1}^{(1)} \oplus Q_{k-1}^{(1)}) }
= \big( \overline{ (F_0^{(2)} \oplus Q_0^{(2)}) \,\sharp\, \cdots \,\sharp\, (F_{k-1}^{(2)}
\oplus Q_{k-1}^{(2)}) } \big)
\circ \Phi \,,
\]
for the diffeomorphism
\[
\Phi: (z_0, \theta_0, \cdots , z_{k-1}, \theta_{k-1}, \zeta_0, \cdots , \zeta_{k-1})
\mapsto (z_0, \theta_0, \cdots , z_{k-1}, \theta_{k-1},
\zeta_0', \cdots , \zeta_{k-1}')
\]
of $\mathbb{R}^{(2n+1)k} \times \mathbb{R}^N$,
where $N = N_0+\cdots+N_{k-1}$,
defined by
\[
\Big( \frac{z_j + z_{j+1}}{2} \,,\, \theta_j \,,\, \zeta_j' \Big)
= \Phi_j \Big(\frac{z_j + z_{j+1}}{2} \,,\, \theta_j \,,\, \zeta_j \Big) \,.
\]
The conjugation $B \,\circ\, \Phi \,\circ\, B^{-1}$
extends to a fibre preserving diffeomorphism
of $S^{2n} \times \mathbb{R} \times \mathbb{R}^{(2n+1)(k-1)} \times \mathbb{R}^N$,
which descends to a fibre preserving diffeomorphism $\overline{\Phi}$
of $S^{2n} \times \mathbb{T}^k \times \mathbb{R}^{2n(k-1)} \times \mathbb{R}^N$,
and we have
\[
\overline{ F_0^{(1)} \,\sharp\, \cdots \,\sharp\, F_{k-1}^{(1)} }
\oplus (Q_0^{(1)} \oplus \cdots \oplus Q_{k-1}^{(1)})
= \overline{ (F_0^{(1)} \oplus Q_0^{(1)}) \,\sharp\, \cdots \,\sharp\,
(F_{k-1}^{(1)} \oplus Q_{k-1}^{(1)}) }
\]
\[
= \overline{ (F_0^{(2)} \oplus Q_0^{(2)}) \,\sharp\, \cdots \,\sharp\, (F_{k-1}^{(2)} \oplus Q_{k-1}^{(2)}) }
\circ \overline{\Phi}
= \Big(\overline{ F_0^{(2)} \,\sharp\, \cdots \,\sharp\, F_{k-1}^{(2)} }
\oplus (Q_0^{(2)} \oplus \cdots \oplus Q_{k-1}^{(2)})\Big)
\circ \overline{\Phi} \,.
\]
This shows that $\overline{ F_0^{(1)} \,\sharp\, \cdots \,\sharp\, F_{k-1}^{(1)} }$
and $\overline{ F_0^{(2)} \,\sharp\, \cdots \,\sharp\, F_{k-1}^{(2)} }$
are equivalent.
\end{proof}


\section{Spectral selectors}\label{section: spectral selectors}

After recalling from \cite{Viterbo, San10, San11}
the definition of the spectral selectors $c_-$ and $c_+$
for compactly supported Hamiltonian diffeomorphisms
of $(\mathbb{R}^{2n}, \omega_0)$
and compactly supported contactomorphisms
of $(\mathbb{R}^{2n} \times S^1, \xi_0)$
contact isotopic to the identity,
we define spectral selectors
detecting flat translated chains
of sequences of compactly supported contactomorphisms
of $(\mathbb{R}^{2n} \times S^1, \xi_0)$
contact isotopic to the identity.
We then prove the properties of these spectral selectors
listed in Proposition \ref{proposition: properties spectral selectors intro}
and Lemma \ref{lemma: bound contact}.

Let $F: E = B \times \mathbb{R}^N \rightarrow \mathbb{R}$
be a function quadratic at infinity over a compact orientable manifold $B$.
Consider the decomposition $E = E^+ \oplus E^- \rightarrow B$
of $E \rightarrow B$ as the direct sum of the vector bundles
$E^+ \rightarrow B$ and $E^- \rightarrow B$
on which the quadratic form $F_{\infty}$
is respectively positive and negative definite.
Denote by $F^-$ the restriction of $F$ to $E^-$.
We consider the sublevel sets
\[
E^{a} = \{\, e \in E \;\lvert\; F(e) \leq a \,\} \,,
\]
and we denote $E^{-\infty} = E^{a}$
for $a$ smaller than all the critical values of $F$.
Similarly, we denote
\[
(E^-)^{-\infty}
= \{\, e \in E^- \;\lvert\; F^- (e) \leq a  \,\}
\]
for $a$ smaller than all the critical values of $F$.
The pair $(E, E^{-\infty})$ deformation retracts to $(E^-, (E^-)^{-\infty})$,
and so we have an isomorphism
\begin{equation}\label{equation: excision}
H^{\ast} \big( D(E^-), S(E^-) \big) \rightarrow H^{\ast} (E^-, (E^-)^{-\infty})
\rightarrow H^{\ast} (E, E^{-\infty}) \,,
\end{equation}
where $D(E^-)$ and $S(E^-)$ are the disc and sphere bundles
of $E^- \rightarrow B$,
and where the first arrow is given by excision.
We denote by 
\[
T: H^{\ast}(B) \rightarrow H^{\ast + \ind(F_{\infty})} (E, E^{-\infty})
\]
the composition of the isomorphism \eqref{equation: excision}
with the Thom isomorphism
\[
H^{\ast}(B) \rightarrow H^{\ast + \ind(F_{\infty})} \big( D(E^-), S(E^-) \big) \,.
\]
For any $a \in \mathbb{R}$ we then consider the homomorphism
\[
i_a^{\ast}: H^{\ast} (E, E^{-\infty}) \rightarrow H^{\ast} (E^{a}, E^{-\infty})
\]
induced by the inclusion
$i_a: (E^{a}, E^{-\infty}) \hookrightarrow (E, E^{-\infty})$,
and the composition
\[
i_a^{\,\ast} \,\circ\, T:
H^{\ast} (B) \rightarrow H^{\ast + \ind (F_{\infty})} (E^{a}, E^{-\infty}) \,.
\]

For $a$ smaller than all the critical values of $F$
the homomorphism $i_a^{\,\ast} \,\circ\, T$ vanishes,
while for $a$ bigger than all the critical values of $F$
it is an isomorphism.
For any non-zero class $u$ in $H^{\ast}(B)$ we define
$$
c(u,F) = \inf \left\{\, a \in \mathbb{R} ~|~
i_a^{\, \ast} \,\circ\, T \, (u) \neq 0 \,\right\} \,.
$$
Then $i_a^{\,\ast} \,\circ\, T \, (u) = 0$
for all $a < c (u, F)$,
and $i_a^{\,\ast} \,\circ\, T \, (u) \neq 0$
for all $a \geq c (u, F)$.
Moreover,
denoting by $1_B \in H^0(B)$ the unit class
and by $\mu_B \in H^{\dim(B)}(B)$
the dual of the fundamental class,
we have the following properties \cite{Viterbo}.

\begin{prop}\label{proposition: properties spectral selectors fqi}
The spectral selectors $c (u, F)$ satisfy the following properties:
\renewcommand{\theenumi}{\roman{enumi}}
\begin{enumerate}
\item \label{properties spectral sel gf: spectrality} \emph{Spectrality:}
$c(u,F)$ is a critical value of $F$.
\item \label{properties spectral sel gf: monotonicity} \emph{Monotonicity:}
if $F_1, F_2: B \times \mathbb{R}^{N} \rightarrow \mathbb{R}$
satisfy $F_1 \leq F_2$
then $c (u, F_1) \leq c (u, F_2)$.
\item \label{properties spectral sel gf: non-degeneracy} \emph{Non-degeneracy:}
if $c (1_B, F) = c (\mu_B, F) = c$ then
the restriction of the projection $E \rightarrow B$
to the set of critical points of $F$ of critical value $c$
is surjective.
\item \label{properties spectral sel gf: equivalence} \emph{Equivalence:}
if $F_1$ and $F_2$ are equivalent then $c (u, F_1) = c (u, F_2)$.
\end{enumerate}
\end{prop}

We will also need the following two lemmas.

\begin{lemma}\label{lemma: continuation}
Let $F_s: E = B \times \mathbb{R}^N \rightarrow \mathbb{R}$,
for $s \in [0, 1]$,
be a $1$-parameter family of functions quadratic at infinity.
Assume that all the functions $F_s$
have the same set of critical values.
For every non-zero class $u$ in $H^{\ast}(B)$
we then have
\[
c (u, F_0) = c (u, F_1) \,.
\]
\end{lemma}

\begin{proof}
Denote $c_s = c (u, F_s)$,
$E_s^{a} = \{ F_s \leq a \}$
and $i_a^{(s)}$ the inclusion
$(E_s^{a}, E^{-\infty}) \hookrightarrow (E, E^{-\infty})$.
By \cite[Lemma 4.3]{FSZ},
for every regular value $a$ of $F_0$
(hence of $F_s$ for all $s$)
there is an isotopy $\{\theta_s\}$ of $E$
such that $\theta_s (E_0^{a}, E^{-\infty}) = (E_s^a, E^{-\infty})$,
and so $(i_a^{(0)})^{\ast} \,\circ\, T \, (u) = 0$
if and only if $(i_a^{(1)})^{\ast} \,\circ\, T \, (u) = 0$.
For every $\epsilon > 0$
such that $c_0 + \epsilon$ is a regular value of $F_0$,
the fact that $(i_{c_0 + \epsilon}^{(0)})^{\ast} \,\circ\, T \, (u) \neq 0$
implies thus that $(i_{c_0 + \epsilon}^{(1)})^{\ast} \,\circ\, T \, (u) \neq 0$.
Since the set of regular values is dense,
this implies that $c_1 \leq c_0$.
Similarly,
inverting the roles of $c_0$ and $c_1$
we obtain $c_0 \leq c_1$,
and so $c_0 = c_1$.
\end{proof}

\begin{lemma}\label{lemma: product}
Let $B_1$ and $B_2$ be compact connected orientable manifolds,
and $f: B_1 \times \mathbb{R}^N \rightarrow \mathbb{R}$
a function quadratic at infinity.
Consider the pullback
$F: B_1 \times B_2 \times \mathbb{R}^N
\rightarrow \mathbb{R}$
of $f$
by the projection
$B_1 \times B_2 \times \mathbb{R}^N
\rightarrow B_1 \times \mathbb{R}^N$.
Then
\[
c (1_{B_1}, f) = c (1_{B_1 \times B_2}, F) \quad \text{ and } \quad
c (\mu_{B_1}, f) = c (\mu_{B_1 \times B_2}, F) \,.
\]
\end{lemma}

\begin{proof}
Denoting by $E_f$ and $E_F$ the domains
of $f$ and $F$ respectively,
and by $(E_f)^{a}$ and $(E_F)^{a}$ the sublevel sets,
we have $E_F = E_f \times B_2$
and $(E_F)^{a} = (E_f)^{a} \times B_2$ for all $a$.
Denote moreover by $(i_a)_f: \big((E_f)^{a}, (E_f)^{-\infty}\big)
\hookrightarrow \big(E_f, (E_f)^{-\infty})$
and $(i_a)_F: \big((E_F)^{a}, (E_F)^{-\infty}\big)
\hookrightarrow \big(E_F, (E_F)^{-\infty})$
the inclusions,
and by $T_f: H^{\ast} (B_1)
\rightarrow H^{\ast + i} \big(E_f, (E_f)^{-\infty}\big)$
and $T_F: H^{\ast} (B_1 \times B_2)
\rightarrow H^{\ast + i} \big(E_F, (E_F)^{-\infty}\big)$,
where $i = \ind (f_{\infty}) = \ind (F_{\infty})$,
the Thom isomorphisms.
Consider the commutative diagram
\[
\begin{tikzcd}
H^{\ast} (B_1)
\arrow[r, "(i_a)_f^{\ast} \,\circ\, T_f"] \arrow{d}
&[4em]  H^{\ast + i} \big( (E_f)^a, (E_f)^{-\infty} \big) \arrow{d} \\
[1em] H^{\ast} (B_1 \times B_2)
\arrow[r, "(i_a)_F^{\ast} \,\circ\, T_F"]
&[4em]  H^{\ast + i} \big( (E_F)^a, (E_F)^{-\infty}\big)
\end{tikzcd}
\]
where the vertical arrows are the homomorphisms
induced by the projections $B_1 \times B_2 \rightarrow B_1$
and
\[
\big( (E_F)^{a}, (E_F)^{-\infty} \big)
= \big( (E_f)^{a} \times B_2, (E_f)^{-\infty} \times B_2 \big)
\rightarrow \big( (E_f)^{a}, (E_f)^{-\infty} \big) \,.
\]
Since the vertical arrow on the left
sends $1_{B_1}$ to $1_{B_1 \times B_2}$
and since the vertical arrow on the right is injective,
we deduce that the image of $1_{B_1}$
by $(i_a)_f^{\,\ast} \,\circ\, T_f$ is non-zero
if and only if the image of $1_{B_1 \times B_2}$
by $(i_a)_F^{\,\ast} \,\circ\, T_F$ is non-zero,
and so $c (1_{B_1}, f) = c (1_{B_1 \times B_2}, F)$.

Similarly,
by considering the commutative diagram
\[
\begin{tikzcd}
H^{m_1} (B_1)
\arrow[r, "(i_a)_f^{\ast} \,\circ\, T_f"] \arrow{d}{\smallsmile \, \mu_{B_2}}
&[4em]  H^{m_1 + i} \big( (E_f)^a, (E_f)^{-\infty} \big) \arrow{d}{\smallsmile \, \mu_{B_2}} \\
[1em] H^{m_1 + m_2} (B_1 \times B_2)
\arrow[r, "(i_a)_F^{\ast} \,\circ\, T_F"]
&[4em]  H^{m_1 + m_2 + i} \big( (E_F)^a, (E_F)^{-\infty}\big)
\end{tikzcd}
\]
where $m_1$ and $m_2$ are the dimensions of $B_1$ and $B_2$ respectively
and where the vertical arrows are the isomorphisms
given by the cup product with $\mu_{B_2}$,
we obtain $c (\mu_{B_1}, f) = c (\mu_{B_1 \times B_2}, F)$,
since $\mu_{B_1} \smallsmile \mu_{B_2} = \mu_{B_1 \times B_2}$.
\end{proof}

For a compactly supported Hamiltonian diffeomorphism
$\varphi$ of $(\mathbb{R}^{2n}, \omega_0)$,
we define
$$
c_- (\varphi) = c (1_{S^{2n}} \,,\, \overline{F}) \quad \text{ and } \quad
c_+ (\varphi) = c (\mu_{S^{2n}} \,,\, \overline{F}) \,,
$$
where $\overline{F}: S^{2n} \times \mathbb{R}^N \rightarrow \mathbb{R}$
is any normalized generating function quadratic at infinity for $\varphi$.
Since the set of critical values of $\overline{F}$
is equal to the action spectrum $\mathcal{A}_{\varphi}$,
we have $c_{\pm} (\varphi) \in \mathcal{A}_{\varphi}$.
Since moreover all normalized
generating functions quadratic at infinity for $\varphi$
are equivalent \cite{Viterbo, Theret},
by Proposition \ref{proposition: properties spectral selectors fqi}
(\ref{properties spectral sel gf: equivalence})
the numbers $c_{\pm} (\varphi)$ do not depend on the choice
of normalized generating function quadratic at infinity for $\varphi$.
In particular,
by Proposition \ref{proposition: gfqi k}
we thus obtain the following lemma.

\begin{lemma}\label{lemma: spectral selectors iteration}
Let $(\varphi_0, \cdots, \varphi_{k-1})$
be a sequence of compactly supported
Hamiltonian diffeomorphisms of $(\mathbb{R}^{2n}, \omega_0)$,
and $F_j$ special generating functions quadratic at infinity
for the $\varphi_j$'s.
Then
\[
c_- \, (\varphi_{k-1} \circ \cdots \circ \varphi_0)
= c \, (1_{S^{2n}} \,,\, \overline{F_0 \,\sharp\, \cdots \,\sharp\, F_{k-1}})
\]
and
\[
c_+ \, (\varphi_{k-1} \circ \cdots \circ \varphi_0)
= c \, (\mu_{S^{2n}} \,,\, \overline{F_0 \,\sharp\, \cdots \,\sharp\, F_{k-1}}) \,.
\]
\end{lemma}

It is proved in \cite{Viterbo} that
$c_- (\varphi) \leq 0 \leq c_+ (\varphi)$,
with equality if and only if 
$\varphi$ is the identity,
and that $c_-(\varphi) = - c_+ (\varphi^{-1})$
(Poincar\'e duality).
We also have the following result,
which will be needed later.

\begin{lemma}\label{lemma: bound symplectic}
For any compactly supported Hamiltonian diffeomorphism
$\varphi$ of $(\mathbb{R}^{2n}, \omega_0)$,
the sets
$\{\, c_+ (\varphi^m) \,,\, m \in \mathbb{Z}_{>0} \,\}$
and $\{\, \lvert\, c_- (\varphi^m) \,\lvert \,,\, m \in \mathbb{Z}_{>0} \,\}$
are bounded from above.
\end{lemma}

\begin{proof}
By Poincar\'e duality,
it is enough to prove that the set
$\{\, c_+ (\varphi^m) \,,\, m \in \mathbb{Z} \smallsetminus \{0\} \,\}$
is bounded from above.
Let $\mathcal{U}$ be an open bounded subset of $\mathbb{R}^{2n}$
containing the support of $\varphi$,
and $\psi$ a compactly supported Hamiltonian diffeomorphism
with $\psi(\mathcal{U}) \cap \mathcal{U} = \emptyset$.
Since, for every $m$,
$\varphi^m$ is supported in $\mathcal{U}$,
by  \cite[Proposition 4.12]{Viterbo}
(or \cite[Lemma 2.13]{San11}) we have
$$
c_+ (\varphi^m) \leq c_+ (\psi) - c_- (\psi) \,.
$$
Since the right hand side is independent of $m$,
we obtain the result.
\end{proof}

Similarly,
for a compactly supported contactomorphism $\phi$
of $(\mathbb{R}^{2n} \times S^1, \xi_0)$
contact isotopic to the identity
we define
\[
c_- (\phi) = c ( 1_{S^{2n} \times S^1} \,,\, \overline{F} )
\quad \text{ and } \quad
c_+ (\phi) = c ( \mu_{S^{2n} \times S^1} \,,\, \overline{F} ) \,,
\]
where $\overline{F}: S^{2n} \times S^1 \times \mathbb{R}^N \rightarrow \mathbb{R}$
is any generating function quadratic at infinity for $\phi$.
Since the set of critical values of $\overline{F}$
is equal to the action spectrum $\mathcal{A}_{\phi}$,
we have $c_{\pm} (\phi) \in \mathcal{A}_{\phi}$.
Since moreover all the generating functions
quadratic at infinity for $\phi$
are equivalent \cite{Viterbo, Theret},
by Proposition \ref{proposition: properties spectral selectors fqi}
(\ref{properties spectral sel gf: equivalence})
the numbers $c_{\pm} (\phi)$
do not depend on the choice
of generating function quadratic at infinity for $\phi$.
It is proved in \cite{San11, San10}
that $c_- (\phi) \leq 0 \leq c_+ (\phi)$,
with equality if and only if $\phi$ is the identity,
and that $c_-(\phi) = - c_+ (\phi^{-1})$.

We now define spectral selectors
detecting flat translated chains
of a sequence $\bm{\phi} = (\phi_0, \cdots, \phi_{k-1})$
of compactly supported contactomorphisms
of $(\mathbb{R}^{2n} \times S^1, \xi_0)$
contact isotopic to the identity.
Suppose first that $k$ is odd.
Let $F_j: \mathbb{R}^{2n+1}\times \mathbb{R}^{N_j} \rightarrow \mathbb{R}$
be special generating functions quadratic at infinity
for the $\phi_j$'s,
and consider the function
\[
\overline{F_0 \,\sharp\, \cdots \,\sharp\, F_{k-1}}:
S^{2n} \times \mathbb{T}^k \times \mathbb{R}^{2n(k-1)} \times \mathbb{R}^N
\rightarrow \mathbb{R}
\]
quadratic at infinity over $S^{2n} \times \mathbb{T}^k$,
where $N = N_0 + \cdots + N_{k-1}$,
constructed in Section \ref{section: functions translated chains}.
We define
$$
c_- (\bm{\phi}) = c (1_{S^{2n} \times \mathbb{T}^k} \,,\, \overline{F_0 \,\sharp\, \cdots \,\sharp\, F_{k-1}})
\quad \text{ and } \quad
c_+ (\bm{\phi}) = c (\mu_{S^{2n} \times \mathbb{T}^k} \,,\,
\overline{F_0 \,\sharp\, \cdots \,\sharp\, F_{k-1}})\,.
$$
In particular,
for $k = 1$ this definition reduces to the one
of the spectral selectors $c_{\pm} (\phi)$
of a contactomorphism $\phi$.
By Proposition \ref{proposition: stabilization etc contact}
and Proposition \ref{proposition: properties spectral selectors fqi}
(\ref{properties spectral sel gf: equivalence})
the numbers $c_{\pm} (\bm{\phi})$ do not depend on the choice
of special generating functions quadratic at infinity
for the $\phi_j$'s.
For a sequence $\bm{\phi} = (\phi_0, \cdots, \phi_{k-1})$
with $k$ even,
we define $c_{\pm} (\bm{\phi}) = c_{\pm} (\bm{\phi}, \id)$.

We now prove the properties of the spectral selectors $c_{\pm}$
listed in Proposition \ref{proposition: properties spectral selectors intro}
and Lemma \ref{lemma: bound contact}.

Spectrality,
Proposition \ref{proposition: properties spectral selectors intro}
(\ref{properties spectral sel intro: spectrality}),
follows for $k$ odd
from Proposition \ref{proposition: quadratic at infinity flat chains}
and Proposition \ref{proposition: properties spectral selectors fqi}
(\ref{properties spectral sel gf: spectrality}).
For $k$ even,
it follows from the odd case
and the fact that
$\mathcal{A}_{(\bm{\phi}, \id)} = \mathcal{A}_{\bm{\phi}}$.

Monotonicity,
Proposition \ref{proposition: properties spectral selectors intro}
(\ref{properties spectral sel intro: monotonicity}),
follows from Proposition \ref{proposition: monotonicity}
and Proposition \ref{proposition: properties spectral selectors fqi}
(\ref{properties spectral sel gf: monotonicity}).
Indeed,
let $\bm{\phi}$ and $\bm{\psi}$ be sequences
with $\bm{\phi} \leq \bm{\psi}$.
By Proposition \ref{proposition: monotonicity},
there are special generating functions quadratic at infinity
$F_j$ and $G_j$ for the $\phi_j$'s and the $\psi_j$'s respectively
such that $F_j \leq G_j$.
We then have
$\overline{F_0 \,\sharp\, \cdots \,\sharp\, F_{k-1}}
\leq \overline{G_0 \,\sharp\, \cdots \,\sharp\, G_{k-1}}$,
and so by Proposition \ref{proposition: properties spectral selectors fqi}
(\ref{properties spectral sel intro: monotonicity})
we conclude that $c_{\pm} (\bm{\phi}) \leq c_{\pm} (\bm{\psi})$.

In order to prove non-degeneracy,
Proposition \ref{proposition: properties spectral selectors intro}
(\ref{properties spectral sel intro: non-degeneracy}),
we need the following lemma.
Recall that a sequence $\bm{\phi} = (\phi_0, \cdots, \phi_{k-1})$
is said to be dynamically trivial
if all the $\phi_j$'s are strict contactomorphisms
and $\phi_{k-1} \circ \cdots \circ \phi_0 = \id$.

\begin{lemma}\label{lemma: dynamically trivial}
A sequence $\bm{\phi} = (\phi_0, \cdots, \phi_{k-1})$
of contactomorphisms of $(\mathbb{R}^{2n+1}, \xi_0)$
is dynamically trivial
if and only if for every $(z_0, \theta_0) \in \mathbb{R}^{2n+1}$
the $k$-tuple
\[
\big( (z_0, \theta_0) \,,\, \phi_0 \, (z_0, \theta_0) \,,\,\cdots \,,\,
\phi_{k-2} \circ \cdots \circ \phi_0 \, (z_0, \theta_0) \big)
\]
is a flat translated chain of $\bm{\phi}$
of action zero.
Moreover, if $\bm{\phi}$ is dynamically trivial
then all the translated chains of $\bm{\phi}$
have action zero.
\end{lemma}

\begin{proof}
The first statement is immediate.
The second can be easily seen
using that strict contactomorphisms
commute with the Reeb flow.
\end{proof}

We can now prove
Proposition \ref{proposition: properties spectral selectors intro}
(\ref{properties spectral sel intro: non-degeneracy}),
whose statement is recalled below.

\begin{prop}\label{lemma: non-degeneracy contact}
A sequence $\bm{\phi}$ of compactly supported contactomorphisms
of $(\mathbb{R}^{2n} \times S^1, \xi_0)$ is dynamically trivial
if and only if $c_- (\bm{\phi}) = c_+ (\bm{\phi}) = 0$.
\end{prop}

\begin{proof}
If $\bm{\phi}$ is dynamically trivial
then $\bm{\phi}_{\mathbb{R}}$ is also dynamically trivial,
and so, by Lemma \ref{lemma: dynamically trivial},
$\mathcal{A}_{\bm{\phi}} = \mathcal{A}_{\bm{\phi}_{\mathbb{R}}} = \{0\}$.
By spectrality we thus conclude that
$c_- (\bm{\phi}) = c_+ (\bm{\phi}) = 0$.
Conversely,
by Proposition \ref{proposition: properties spectral selectors fqi}
(\ref{properties spectral sel gf: non-degeneracy})
we have that $c_- (\bm{\phi}) = c_+ (\bm{\phi}) = 0$
if and only if the restriction of the projection
$$
S^{2n} \times \mathbb{T}^k \times \mathbb{R}^{2n (k-1)} \times \mathbb{R}^N
\rightarrow S^{2n} \times \mathbb{T}^k
$$
to the set of critical points
of $\overline{F_0 \,\sharp\,\cdots\,\sharp\,F_{k-1}}$
of critical value zero
is surjective,
and thus if and only if the restriction of the projection
\begin{equation}\label{equation: projection}
\mathbb{R}^{(2n+1)k} \times \mathbb{R}^N
\rightarrow \mathbb{R}^{2n} \times \mathbb{R}^k \,,
\end{equation}
\[
(z_0, \theta_0, \cdots, z_{k-1}, \theta_{k-1},
\zeta_0, \cdots, \zeta_{k-1})
\rightarrow (z_0, \theta_0, \cdots, \theta_{k-1})
\]
to the set of critical points of $F_0 \,\sharp\,\cdots\,\sharp\,F_{k-1}$
of critical value zero is surjective.
Suppose that this is the case.
Then for any $(z_0, \theta_0) \in \mathbb{R}^{2n+1}$,
posing $\theta_1 = ((\phi_0)_{\mathbb{R}})_{\theta} \, (z_0, \theta_0)$,
$\cdots$,
$\theta_{k-1} =
((\phi_{k-2})_{\mathbb{R}} \circ \cdots \circ (\phi_0)_{\mathbb{R}})_{\theta} \, (z_0, \theta_0)$,
there are $z_1$, $\cdots$, $z_{k-1}$, $\zeta_0, \cdots, \zeta_{k-1}$
such that
$$
(z_0, \theta_0, \cdots, z_{k-1}, \theta_{k-1}, \zeta_0, \cdots, \zeta_{k-1})
$$
is a critical point of $F_0 \,\sharp\,\cdots\,\sharp\,F_{k-1}$
of critical value zero.
By Proposition \ref{proposition: composition formula contact flat},
$$
\big((z_0, \theta_0), \cdots, (z_{k-1}, \theta_{k-1})\big)
$$
is a flat translated chain of $\bm{\phi}_{\mathbb{R}}$
of action zero.
By our choice of $\theta_1, \cdots, \theta_{k-1}$,
this translated chain is equal to
\[
\big( (z_0, \theta_0) \,,\, \phi_0 \, (z_0, \theta_0) \,,\,\cdots \,,\,
\phi_{k-2} \circ \cdots \circ \phi_0 \, (z_0, \theta_0) \big) \,.
\]
By Lemma \ref{lemma: dynamically trivial},
we thus conclude that $\bm{\phi}_{\mathbb{R}}$
is dynamically trivial,
and so $\bm{\phi}$ is dynamically trivial.
\end{proof}

We  now prove Proposition \ref{proposition: properties spectral selectors intro}
(\ref{properties spectral sel intro: relation symplectic}),
whose statement is recalled below.

\begin{prop}\label{proposition: relation with symplectic}
For any sequence
$\bm{\varphi} = (\varphi_0, \cdots, \varphi_{k-1})$
of compactly supported Hamiltonian diffeomorphisms
of $(\mathbb{R}^{2n}, \omega_0)$,
we have
$$
c_{\pm} (\widetilde{\bm{\varphi}})
= c_{\pm} (\varphi_{k-1} \circ \cdots \circ \varphi_0) \,.
$$
\end{prop}

\begin{proof}
Let $f_j: \mathbb{R}^{2n} \times \mathbb{R}^{N_j} \rightarrow \mathbb{R}$
be special generating functions quadratic at infinity
for the $\varphi_j$'s.
By Lemma \ref{lemma: gf lift to contact},
the functions $F_j: \mathbb{R}^{2n+1} \times \mathbb{R}^{N_j}
\rightarrow \mathbb{R}$
defined by
\[
F_j (z, \theta, \zeta) = f_j (z, \zeta)
\]
are special generating functions quadratic at infinity
for the $\widetilde{\varphi_j}'s$.
We have
$$
\overline{F_0 \,\sharp\, \cdots \,\sharp\, F_{k-1}}
= \overline{f_0 \,\sharp\, \cdots \,\sharp\, f_{k-1}} \circ \pr \,,
$$
where
\[
\pr: S^{2n} \times \mathbb{T}^k \times \mathbb{R}^{2n(k-1)} \times \mathbb{R}^{N}
\rightarrow S^{2n} \times \mathbb{R}^{2n(k-1)} \times \mathbb{R}^{N}
\]
is the projection that forgets the coordinates
$(\theta_0, \cdots, \theta_{k-1}) \in \mathbb{T}^k$.
By Lemma \ref{lemma: product}
and Lemma \ref{lemma: spectral selectors iteration}
we thus have
\[
c_- (\widetilde{\bm{\varphi}})
= c (1_{S^{2n} \times \mathbb{T}^k} \,,\, \overline{F_0 \,\sharp\, \cdots \,\sharp\, F_{k-1}})
= c (1_{S^{2n}} \,,\, \overline{f_0 \,\sharp\, \cdots \,\sharp\, f_{k-1}})
= c_- \, (\varphi_{k-1} \circ \cdots \circ \varphi_0)
\]
and
\[
c_+ (\widetilde{\bm{\varphi}})
= c (\mu_{S^{2n} \times \mathbb{T}^k} \,,\, \overline{F_0 \,\sharp\, \cdots \,\sharp\, F_{k-1}})
= c (\mu_{S^{2n}} \,,\, \overline{f_0 \,\sharp\, \cdots \,\sharp\, f_{k-1}})
= c_+ \, (\varphi_{k-1} \circ \cdots \circ \varphi_0) \,.
\qedhere
\]
\end{proof}

Proposition \ref{proposition: properties spectral selectors intro}
(\ref{properties spectral sel intro: monotonicity})
and (\ref{properties spectral sel intro: relation symplectic}) 
imply Proposition \ref{proposition: properties spectral selectors intro}
(\ref{properties spectral sel intro: sign}):
for any sequence $\bm{\phi}$,
we have $c_-(\bm{\phi}) \leq 0$ and $c_+(\bm{\phi}) \geq 0$.  
Indeed,
let $\bm{\varphi}_1$ and $\bm{\varphi}_2$
be sequences of compactly supported Hamiltonian diffeomorphisms
of $(\mathbb{R}^{2n}, \omega_0)$
such that $\widetilde{\bm{\varphi}_1} \leq \bm{\phi} \leq \widetilde{\bm{\varphi}_2}$.
By Proposition \ref{proposition: properties spectral selectors intro}
(\ref{properties spectral sel intro: monotonicity})
and (\ref{properties spectral sel intro: relation symplectic}),
we have
\[
c_- (\bm{\phi}) \leq c_- (\widetilde{\bm{\varphi}_2})
= c_- \big( (\varphi_2)_{k-1} \circ \cdots \circ (\varphi_2)_0 \big) \leq 0
\]
and
\[
c_+ (\bm{\phi}) \geq c_+ (\widetilde{\bm{\varphi}_1})
= c_+ \big( (\varphi_1)_{k-1} \circ \cdots \circ (\varphi_1)_0 \big) \geq 0 \,.
\]

Similarly,
Proposition \ref{proposition: properties spectral selectors intro}
(\ref{properties spectral sel intro: monotonicity})
and (\ref{properties spectral sel intro: relation symplectic})
and Lemma \ref{lemma: bound symplectic}
imply that, for any sequence $\bm{\phi}$,
the sets $\{\, c_+ (\bm{\phi}^m) \,,\, m \in \mathbb{Z}_{>0} \,\}$
and $\{\, \lvert \, c_- (\bm{\phi}^m) \,\lvert \,,\, m \in \mathbb{Z}_{>0} \,\}$
are bounded from above
(Lemma \ref{lemma: bound contact}).
Indeed,
let $\bm{\varphi}_1$ and $\bm{\varphi}_2$ be sequences
of compactly supported Hamiltonian diffeomorphisms
of $(\mathbb{R}^{2n}, \omega_0)$ with
$\widetilde{\bm{\varphi}_1} \leq \bm{\phi} \leq \widetilde{\bm{\varphi}_2}$.
By Proposition \ref{proposition: properties spectral selectors intro}
(\ref{properties spectral sel intro: monotonicity})
and (\ref{properties spectral sel intro: relation symplectic})
we have
$$
c_+ (\bm{\phi}^m) \leq c_+ \big( (\widetilde{\bm{\varphi}_2})^m\big)
= c_+ \Big( \big( (\varphi_2)_{k-1} \circ \cdots \circ (\varphi_2)_0 \big)^m \Big)
$$
and
$$
c_- (\bm{\phi}^m) \geq c_- \big(\widetilde{\bm{\varphi}_1})^m\big)
= c_- \Big( \big( (\varphi_1)_{k-1} \circ \cdots \circ (\varphi_1)_0 \big)^m \Big) \,,
$$
and so the required bounds follow
from Lemma \ref{lemma: bound symplectic}.

Finally we prove stabilization,
Proposition \ref{proposition: properties spectral selectors intro}
(\ref{properties spectral sel intro: decreasing sequence}),
whose statement is recalled below.

\begin{prop}\label{proposition: decreasing sequence}
For any sequence $\bm{\phi}$ of compactly supported contactomorphisms
of $(\mathbb{R}^{2n} \times S^1, \xi_0)$
contact isotopic to the identity,
we have
\[
c_{\pm} (\bm{\phi}) = c_{\pm} (\bm{\phi}, \id) \,.
\]
\end{prop}

\begin{proof}
It is enough to prove that 
\[
c_{\pm} (\bm{\phi}) = c_{\pm} (\bm{\phi}, \id, \id)
\]
for a sequence $\bm{\phi} = (\phi_0, \cdots, \phi_{k-1})$
with $k$ odd.
Let $F_j: \mathbb{R}^{2n+1} \times \mathbb{R}^{N_j} \rightarrow \mathbb{R}$
be special generating functions quadratic at infinity
for the $\phi_j$'s.
We have
\[
(F_0 \,\sharp\, \cdots \,\sharp\, F_{k-1} \,\sharp\, 0 \,\sharp\, 0) \circ B^{-1} \,
(z_0, \theta_0, \cdots, z_{k+1}, \theta_{k+1},
\zeta_0, \cdots, \zeta_{k-1})
\]
\[
= F_0 \Big( z_0 + \frac{z_1}{2} \,,\, \theta_0 \,,\, \zeta_0 \Big)
+ \sum_{j = 1}^{k-1}
F_j \Big( z_0 + \frac{z_j + z_{j+1}}{2} \,,\,
\theta_j \,,\, \zeta_j \Big)  
+ \sum_{j = 1}^{k} \frac{1}{2} \, \langle\, i \, z_j , z_{j+1} \,\rangle \,.
\]
Consider the $1$-parameter family of functions
$$
G_s: \mathbb{R}^{(2n+1)(k+2)} \times \mathbb{R}^{N_0 + \cdots + N_{k-1}}
\rightarrow \mathbb{R} \,,
$$
for $s \in [0,1]$,
defined by
\[
G_s (z_0, \theta_0, \cdots, z_{k+1}, \theta_{k+1},
\zeta_0, \cdots, \zeta_{k-1})
\]
\[
= F_0 \Big( z_0 + \frac{z_1}{2} \,,\, \theta_0 \,,\, \zeta_0 \Big)
+ \sum_{j = 1}^{k - 2}
F_j \Big( z_0 + \frac{z_j + z_{j+1}}{2} \,,\, \theta_j \,,\, \zeta_j \Big)
+ F_{k-1} \Big( z_0 + \frac{z_{k-1} + s z_{k}}{2} \,,\,
\theta_{k-1} \,,\, \zeta_{k-1} \Big)
\]
\[
+ \sum_{j = 1}^{k-2} \frac{1}{2} \, \langle\, i \, z_j , z_{j+1} \,\rangle 
+ \frac{s}{2} \, \langle\, i \, z_{k-1}, z_k \,\rangle
+ \frac{1}{2} \, \langle\, i \, z_k, z_{k+1} \,\rangle \,.
\]
Since $G_0$ is a stabilization of $F_0 \,\sharp \,\cdots \,\sharp\, F_{k-1}$
(i.e.\ the direct sum of $F_0 \,\sharp \,\cdots \,\sharp\, F_{k-1}$
with a non-degenerate quadratic form),
its set of critical values
is equal to the set of critical values of $F_0 \,\sharp \,\cdots \,\sharp\, F_{k-1}$,
which is equal to $\mathcal{A}_{\bm{\phi}}$.
On the other hand,
the set of critical values of
$G_1 = F_0 \,\sharp \,\cdots \,\sharp\, F_{k-1} \,\sharp\, 0 \,\sharp\, 0$
is equal to $\mathcal{A}_{(\bm{\phi}, \id, \id)}$,
which is also equal to $\mathcal{A}_{\bm{\phi}}$.
We claim that, in fact, all the functions $G_s$
have the same set of critical values.
Observe first that for any critical point
\[
p = (z_0, \theta_0, \cdots, z_{k+1}, \theta_{k+1},
\zeta_0, \cdots, \zeta_{k-1})
\]
of $G_s$ we have
\begin{equation}\label{equation: xy}
z_k = z_{k+1} = 0 \,.
\end{equation}
Indeed,
on the one hand we have
\[
0 = \frac{\partial G_s}{\partial z_{k+1}} \, (p)
= \frac{i}{2} \, z_k \,,
\]
thus $z_k = 0$.
On the other hand,
the vanishing of $z_{k+1}$
can be seen as follows.
Denote
\[
D_0 = d_z F_0 \, \Big( z_0 + \frac{z_1}{2} \,,\, \theta_0 \,,\, \zeta_0 \Big) \,,
\]
\[
D_j = d_z F_j \, \Big( z_0 + \frac{z_j + z_{j+1}}{2} \,,\, \theta_j \,,\, \zeta_j \Big)
\]
for $j = 1, \cdots, k-2$
and
\[
D_{k-1} = d_z F_{k-1} \, \Big( z_0 + \frac{z_{k-1} + s z_{k}}{2} \,,\,
\theta_{k-1} \,,\, \zeta_{k-1} \Big) \,.
\]
For $k = 3$,
the equalities
\[
0 = \frac{\partial G_s}{\partial z_0} \, (p) = D_0 + D_1 + D_2 \,,
\]
\[
0 = \frac{\partial G_s}{\partial z_1} \, (p)
= \frac{D_0}{2} + \frac{D_1}{2} - \frac{i}{2} \, z_2
\]
and
\[
0 = \frac{\partial G_s}{\partial z_3} \, (p)
= \frac{s D_2}{2} + \frac{i}{2} (s z_2 - z_4)
\]
imply that
\[
0 = s \, \frac{\partial G_s}{\partial z_1} \, (p)
+ \frac{\partial G_s}{\partial z_3} (p)
= - \frac{i}{2} \, z_4 \,,
\]
thus $z_4 = 0$.
For $k \geq 5$,
the equalities
\[
0 = \frac{\partial G_s}{\partial z_0} \, (p)
= D_0 + \cdots + D_{k-1} \,,
\]
\[
0 = \frac{\partial G_s}{\partial z_1} \, (p)
= \frac{D_0}{2} + \frac{D_1}{2} - \frac{i}{2} \, z_2 \,,
\]
\[
0 = \frac{\partial G_s}{\partial z_{2j-1}} \, (p)
= \frac{D_{2j-2}}{2}
+ \frac{D_{2j-1}}{2} - \frac{i}{2} (z_{2j} - z_{2j-2} )
\]
for $j = 2, \cdots, \frac{k-1}{2}$,
and
\[
0 = \frac{\partial G_s}{\partial z_{k}} (p)
= \frac{s D_{k-1}}{2} + \frac{i}{2} (s z_{k-1} - z_{k+1})
\]
imply that
\[
0 = s \; \frac{\partial G_s}{\partial z_1} (p)
+ s \; \frac{\partial G_s}{\partial z_3} (p)
+ \cdots + s \; \frac{\partial G_s}{\partial z_{k-2}} (p)
+ \frac{\partial G_s}{\partial z_{k}} (p)
= - \frac{i}{2} \, z_{k+1} \,,
\]
thus $z_{k+1} = 0$.

Denote now
\[
p_s = (z_0, \theta_0, \cdots, z_{k-1}, \theta_{k-1},
s z_k, \theta_k, z_{k+1}, \theta_{k+1},
\zeta_0, \cdots, \zeta_{k-1}) \,.
\]
Then
\begin{equation}\label{equation: G}
G_s (p) = G_1 (p_s) + \frac{1-s}{2} \, \langle\, i \, z_k \,,\, z_{k+1} \,\rangle \,.
\end{equation}
Using \eqref{equation: xy} we see that
\[
0 = \frac{\partial G_s}{\partial z_k} \, (p)
= s \; \frac{\partial G_1}{\partial z_k} \, (p_s) - \frac{i \, (1-s)}{2} \, z_{k+1}
= s \; \frac{\partial G_1}{\partial z_k} \, (p_s)
\]
and
\[
0 = \frac{\partial G_s}{\partial z_{k+1}} \, (p)
= \frac{\partial G_1}{\partial z_{k+1}} \, (p_s) + \frac{i\,(1-s)}{2} \, z_k
= \frac{\partial G_1}{\partial z_{k+1}} \, (p_s) \,.
\]
Thus (assuming that $s > 0$),
$\frac{\partial G_1}{\partial z_k} \, (p_s)
= \frac{\partial G_1}{\partial z_{k+1}} \, (p_s) = 0$.
Since all the other derivatives of $G_1$ at $p_s$
are equal to those of $G_s$ at $p$,
we conclude that $p_s$ is a critical point of $G_1$.
In fact, the same calculations show that $p$ is a critical point of $G_s$
if and only if $p_s$ is a critical point of $G_1$.
By \eqref{equation: xy} and \eqref{equation: G},
we then also have $G_s (p) = G_1 (p_s)$.
We thus conclude that all the functions $G_s$
have the same set of critical values,
as we claimed.

We now compactify the functions $G_s$
to functions quadratic at infinity
$$
\overline{G_s}: S^{2n} \times \mathbb{T}^k \times \mathbb{R}^{N_0 + \cdots + N_{k-1}}
\rightarrow \mathbb{R} \,.
$$
The functions $\overline{G_s}$ also have the same set of critical values.
By Proposition \ref{proposition: properties spectral selectors fqi}
(\ref{properties spectral sel gf: equivalence})
(using that $\overline{G_0}$ is a stabilization
of $\overline{F_0 \,\sharp\,\cdots \,\sharp\, F_{k-1}}$)
and Lemma \ref{lemma: continuation}
we thus have
\[
c_+ (\bm{\phi})
= c (\mu_{S^{2n} \times \mathbb{T}^k} \,,\,
\overline{F_0 \,\sharp\,\cdots \,\sharp\, F_{k-1}})
= c (\mu_{S^{2n} \times \mathbb{T}^k} \,,\, \overline{G_0})
\]
\[
= c (\mu_{S^{2n} \times \mathbb{T}^k} \,,\, \overline{G_1})
= c_+ (\bm{\phi}, \id, \id) \,.
\]
Similarly,
$c_- (\bm{\phi}) = c_- (\bm{\phi}, \id, \id)$.
\end{proof}


\section{Proof of Theorem \ref{theorem: main}}\label{section: proof}

We first remark that it is enough to prove Theorem \ref{theorem: main}
in the case of $(\mathbb{R}^{2n} \times S^1, \xi_0)$.
Indeed,
for any sequence $\bm{\phi} = (\phi_0, \cdots,\phi_{k-1})$
of compactly supported contactomorphisms of $(\mathbb{R}^{2n+1}, \xi_0)$
contact isotopic to the identity
there exists $a > 0$ such that each $\phi_j$
is the time-$1$ map of a contact isotopy
supported in $\R^{2n} \times (- \frac{a}{2} , \frac{a}{2})$.
We consider the contactomorphism $\nu$
of $(\mathbb{R}^{2n+1}, \xi_0)$
given by
\begin{align*}
\nu (z, \theta) = (\sqrt{a} \, z \,,\, a \, \theta) \,,
\end{align*}
and the sequence
\[
\bm{\phi}_a := (\nu^{-1} \circ \varphi_0 \circ \nu \,,\, \cdots \,,\,
\nu^{-1} \circ \varphi_{k-1} \circ \nu) \,.
\]
The periodic flat translated chains of $\bm{\phi}$
are in bijection with the periodic flat translated chains of $\bm{\phi}_a$.
Since each $\nu^{-1} \circ \varphi_j \circ \nu$
is the time-$1$ map of a contact isotopy
supported in $\mathbb{R}^{2n} \times (- \frac{1}{2}, \frac{1}{2})$,
we can see $\bm{\phi}_a$ as a sequence
of compactly supported contactomorphisms
of $(\mathbb{R}^{2n} \times S^1, \xi_0)$
contact isotopic to the identity.

By arguments parallel to those given in \cite{Viterbo}
to prove Theorem \ref{theorem: Viterbo},
we now use Lemma \ref{lemma: action relation for not geometrically distinct chains},
Proposition \ref{proposition: properties spectral selectors intro}
and Lemma \ref{lemma: bound contact}
to prove Theorem \ref{theorem: main}
for $(\mathbb{R}^{2n} \times S^1, \xi_0)$.
We first show that any non-trivial sequence
$\bm{\phi} = (\phi_0, \cdots, \phi_{k-1})$
of compactly supported contactomorphisms
of $(\mathbb{R}^{2n} \times S^1, \xi_0)$
contact isotopic to the identity
has infinitely many geometrically distinct
periodic flat translated chains
in the interior of the support.
As remarked in the introduction,
we can assume that 
$\bm{\phi}$ is not dynamically trivial.
Let $M_{\bm{\phi}}$ be the set of positive integers $m$
such that $\bm{\phi}^m$ is not dynamically trivial.
Then $M_{\bm{\phi}}$ is infinite.
Indeed,
its complement cannot contain any two consecutive numbers $l$ and $l+1$:
if $\bm{\phi}^l$ is dynamically trivial then $\bm{\phi}^{l+1}$
cannot be dynamically trivial,
because this would imply that $\bm{\phi}$
is dynamically trivial.
By Proposition \ref{proposition: properties spectral selectors intro}
(\ref{properties spectral sel intro: sign})
and (\ref{properties spectral sel intro: non-degeneracy}),
for any $m$ in $M_{\bm{\phi}}$
we have that either $c_+ (\bm{\phi}^m) > 0$
or $c_- (\bm{\phi}^m) < 0$.
Thus, there is a sequence $l(j) \to \infty$
in $M_{\bm{\phi}}$
such that either $c_+ (\phi^{l(j)}) > 0$ for all $j$
or $c_- (\phi^{l(j)}) < 0$ for all $j$.
We consider the case where $c_+ (\phi^{l(j)}) > 0$ for all $j$,
the other case being similar.
Set $c_{l(j)} = c_+(\bm{\phi}^{l(j)})$.
By Proposition \ref{proposition: properties spectral selectors intro}
(\ref{properties spectral sel intro: spectrality}),
there is a flat translated chain $\bm{p}_{l(j)}$
of $\bm{\phi}^{l(j)}$ with action $c_{l(j)}$.
By Lemma \ref{lemma: bound contact},
the sequence $(c_{l(j)})$ is bounded from above.
Thus $\frac{c_{l(j)}}{l(j)}$ tends to zero as $j \to \infty$,
and we can define a sequence $\big(m(j)\big)$
of positive numbers inductively by posing $m(1) = l(1)$
and $m (j+1)$ the smallest element $m$
of the sequence $\big(l(j)\big)$
that is bigger than $m(j)$ and satisfies
$m > m(j) \cdot \frac{c_m}{c_{m(j)}}$.
By construction,
$\frac{c_{m(j_1)}}{m(j_1)} \neq \frac{c_{m(j_2)}}{m(j_2)}$
for all $j_1 \neq j_2$.
By Lemma \ref{lemma: action relation for not geometrically distinct chains},
we thus conclude that all the periodic flat translated chains of $\bm{\phi}$
in the set $\{\, \bm{p}_{m(j)} \;\lvert\; j \in \mathbb{Z}_{>0} \,\}$
are geometrically distinct.

Finally we prove that if $\bm{\phi}$ is non-negative
then
\[
\limsup_{m \to \infty} \, \frac{N (\bm{\phi}, m)}{m} > 0 \,,
\]
where $N (\bm{\phi}, m)$ denotes the number
of geometrically distinct periodic flat translated chains of $\bm{\phi}$
contained in the interior of the support
and having period smaller than or equal to $m$.
Since $\bm{\phi}$ is non-negative,
Proposition \ref{proposition: properties spectral selectors intro}
(\ref{properties spectral sel intro: spectrality}),
(\ref{properties spectral sel intro: sign})
and (\ref{properties spectral sel intro: monotonicity})
imply that $c_- (\bm{\phi}) = 0$.
By Proposition \ref{proposition: properties spectral selectors intro}
(\ref{properties spectral sel intro: sign})
and (\ref{properties spectral sel intro: non-degeneracy})
we thus have $c_+ (\bm{\phi}) > 0$.
Moreover,
for every positive integer $m$
we also have $c_+ (\bm{\phi}^m) > 0$.
Indeed, since
\[
(\bm{\phi}, \bm{\id}, \cdots, \bm{\id}) \leq \bm{\phi}^m \,,
\]
by Proposition \ref{proposition: properties spectral selectors intro}
(\ref{properties spectral sel intro: monotonicity})
and (\ref{properties spectral sel intro: decreasing sequence})
we have
\[
0 < c_+ (\bm{\phi}) = c_+ (\bm{\phi}, \bm{\id}, \cdots, \bm{\id})
\leq c_+ (\bm{\phi}^m) \,.
\]
Set $c_m = c_+ (\bm{\phi}^m)$.
By Proposition \ref{proposition: properties spectral selectors intro}
(\ref{properties spectral sel intro: spectrality})
there is a flat translated chain $\bm{p}_m$ of $\bm{\phi}^m$
with action $c_m$.
By Lemma \ref{lemma: bound contact},
the sequence $(c_m)$ is bounded from above,
thus $\frac{c_m}{m}$ tends to zero as $m \to \infty$,
and we can define a sequence $\big(m(j)\big)$
of positive numbers inductively by posing $m(1) = 1$
and $m (j+1)$ the smallest integer
that is bigger than $m(j)$
and satisfies $m > m (j) \cdot \frac{c_m}{c_{m(j)}}$.
By Lemma \ref{lemma: action relation for not geometrically distinct chains},
all the periodic flat translated chains of $\bm{\phi}$
in the set $\{\, \bm{p}_{m(j)} \;\lvert\; j \in \mathbb{Z}_{>0} \,\}$
are geometrically distinct.
Thus,
\begin{equation}\label{equation: in proof theorem}
N \big(\bm{\phi}, m(j)\big) \geq j \,.
\end{equation}
Moreover,
by definition of the sequence $\big(m(j)\big)$
we have
\[
m (j + 1) - 1 \leq m (j) \,\cdot\, \frac{c_{m (j + 1) - 1}}{c_{m(j)}} \,.
\]
But, by Proposition \ref{proposition: properties spectral selectors intro}
(\ref{properties spectral sel intro: monotonicity})
and (\ref{properties spectral sel intro: decreasing sequence}),
$c_{m (j + 1) - 1} \leq c_{m (j + 1)}$.
We thus obtain
\[
m (j + 1) - 1 \leq m (j) \,\cdot\, \frac{c_{m (j + 1)}}{c_{m(j)}} \,,
\]
and so
\[
\frac{m (j + 1)}{m(j)} \leq \frac{c_{m (j + 1)}}{c_{m(j)}} + \frac{1}{m(j)}
\leq \frac{c_{m (j + 1)}}{c_{m(j)}} + \frac{1}{j}
\leq \frac{c_{m (j + 1)}}{c_{m(j)}} \, \Big( 1 + \frac{1}{j} \Big) \,.
\]
Let $C$ be a real number such that $c_{m(j)} \leq C$ for all $j$.
Then
\[
m(j) = \prod_{l = 1}^{j - 1} \, \frac{m (l+1)}{m (l)}
\leq \prod_{l = 1}^{j - 1} \, \frac{c_{m (l + 1)}}{c_{m(l)}} \, \Big(1 + \frac{1}{l} \Big)
= \frac{c_{m(j)}}{c_{m(1)}} \,\cdot\, j
\leq \frac{C}{c_{m(1)}} \,\cdot\, j \,.
\]
By \eqref{equation: in proof theorem},
and since $c_{m(1)} = c_1 = c_+ (\bm{\phi})$,
we thus obtain
\[
N \big(\bm{\phi}, m(j)\big) \geq j \geq m(j) \,\cdot\, \frac{c_+(\bm{\phi})}{C} \,,
\]
and so 
\[
\limsup_{m \to \infty} \, \frac{N (\bm{\phi}, m)}{m} > 0 \,.
\]


\end{document}